\documentclass[12pt,french,brochure]{bourbaki}

\usepackage{epsfig,graphics}
\usepackage{graphicx}
\usepackage[utf8]{inputenc}
\usepackage{eucal}
\usepackage[obeyspaces]{url}
\usepackage{babel}
\usepackage{pstricks}
\renewcommand{\H}{\mathbf{H}}
\newcommand{\Q}{\mathbf{Q}}
\newcommand{\N}{\mathbf{N}}
\newcommand{\R}{\mathbf{R}}
\newcommand{\T}{\mathbf{T}}
\newcommand{\bP}{\mathbf{P}}
\newcommand{\C}{\mathbf{C}}

\newcommand{\Z}{\mathbf{Z}}
\newcommand{\SL}{\mathrm{SL}}
\newcommand{\PSL}{\mathrm{PSL}}

\newtheorem{ques}[defi]{Question}

\date{Juin 2012}
\bbkannee{64\`eme ann\'ee, 2011-12} 
\bbknumero{1055}
\title{La conjecture des sous-groupes de surfaces}
\subtitle{d'apr\`es Jeremy Kahn et Vladimir Markovic}

\author{Nicolas BERGERON}
\address{Universit\'e Pierre et Marie Curie et \\ Institut Universitaire de France\\
4 place Jussieu \\
F-75005 PARIS}
\email{bergeron@math.jussieu.fr}

\begin{document}
\maketitle

Une vari\'et\'e hyperbolique est une vari\'et\'e riemannienne, lisse, connexe, compl\`ete et de courbure sectionnelle constante \'egale \`a $-1$. On note $\H_3$ l'unique -- \`a isom\'etrie
pr\`es -- vari\'et\'e hyperbolique simplement connexe de dimension $3$. Dans ce rapport, on utilise le mod\`ele du demi-espace $\H_3 = (\C \times \R_+^* , \frac{|dz|^2 + dt^2}{t^2})$.
L'extension de Poincar\'e de l'action par homographies de $\PSL_2 (\C)$ sur $\bP_1(\C)$ fournit une action de $\PSL_2 (\C)$ sur $\H_3$ par isom\'etries et donc -- {\it via} la diff\'erentielle -- une action sur le fibr\'e des rep\`eres. Cette action est simplement transitive; le choix d'un point base
$p_0=(0,1) \in \H_3$ et de deux vecteurs tangents unitaires orthogonaux
$\vec{u}_0 = (0,1)$ et $\vec{n}_0 =(1,0)$ permet donc d'identifier $\PSL_2 (\C)$ au fibr\'e des rep\`eres de $\H_3$ {\it via} l'application $g \mapsto g \cdot (p_0 , \vec{u}_0, \vec{n}_0 , 
\vec{u}_0 \wedge \vec{n}_0)$.
 
Sauf mention contraire, dans tout ce rapport le terme {\it vari\'et\'e} d\'esignera dor\'enavant une vari\'et\'e lisse, compacte, connexe, orientable et sans bord, et {\it surface} une vari\'et\'e de dimension $2$. Un choix indiff\'erent de point base est effectu\'e lorsque l'on parle du groupe fondamental d'un espace connexe par arc, et un choix appropri\'e de points bases est effectu\'e lorsque l'on parle de morphisme entre groupes fondamentaux.   

\section{Introduction : la conjecture VH et ses amies}

Soit $M$ une vari\'et\'e hyperbolique de dimension $3$. On conjecture tour \`a tour les \'enonc\'es suivants (voir \cite[p. 380]{Thurston})~:

\begin{enumerate}
\item Le groupe fondamental de $M$ contient un sous-groupe isomorphe au groupe fondamental d'une surface de genre au moins $2$. 
\item La vari\'et\'e $M$ poss\`ede un rev\^etement fini qui contient une surface {\it plong\'ee} de sorte que l'inclusion induise une injection au niveau des groupes fondamentaux. 
\item La vari\'et\'e $M$ poss\`ede un rev\^etement fini dont le premier nombre de Betti est non nul.
\item Pour tout entier $n$, la vari\'et\'e $M$ poss\`ede un rev\^etement fini dont le premier nombre de Betti est sup\'erieur \`a $n$.
\item La vari\'et\'e $M$ poss\`ede un rev\^etement fini dont le groupe fondamental se surjecte sur un groupe libre de rang $2$. 
\item La vari\'et\'e $M$ poss\`ede un rev\^etement fini qui fibre sur le cercle.
\end{enumerate}

Il n'est pas difficile de v\'erifier que\footnote{L'implication $(2) \Rightarrow (1)$ est la plus d\'elicate,
voir \cite[Lem. 6.6]{Hempel}.} 
$$\begin{array}{rcl}
(5) \Rightarrow (4)  \Rightarrow & (3) & \Rightarrow (2) \Rightarrow (1) \\
& \Uparrow  & \\
& (6) & 
\end{array} $$

La conjecture VH (pour virtuellement Haken) est l'\'enonc\'e (2). 
L'objet de ce rapport est d'expliquer les grandes id\'ees de la d\'emonstration, par Jeremy Kahn et Vladimir Markovic \cite{KM1}, de la conjecture (1)~:

\begin{theo}[Kahn-Markovic] \label{T1}
Soit $M$ une vari\'et\'e hyperbolique de dimension $3$, alors il existe une surface $S$ de genre $g \geq 2$ et une injection 
$$\pi_1 S \hookrightarrow \pi_1 M.$$
\end{theo}

\subsubsection*{Commentaires}

1. Dans une pr\'epublication r\'ecente, Dani Wise d\'emontre notamment l'implication $(2) \Rightarrow (6)$. Son r\'esultat est plus g\'en\'eral, nous revenons sur ses travaux au \S \ref{par:6.1}. Tout r\'ecemment Ian Agol
a fait circuler une pr\'epublication dans laquelle, en se reposant sur les travaux de Wise et de Kahn et Markovic, il d\'emontre les six conjectures ci-dessus.

2. Dans le cas o\`u $M$ est {\it arithm\'etique}, le th\'eor\`eme \ref{T1} est d\^u \`a Marc Lackenby \cite{Lackenby}.

3. Dans le cas o\`u $M$ est une vari\'et\'e hyperbolique de dimension $3$ de volume fini mais {\it non compacte}, Daryl Cooper, Darren Long et Alan Reid \cite{CLR} d\'emontrent (5) et les r\'esultats
de Wise s'appliquent encore pour d\'emontrer (6).

4. Dans un cadre plus g\'en\'eral, Mikhail Gromov pose~:

\begin{ques}
Soit $\Gamma$ un groupe hyperbolique au sens de Gromov qui ne contient pas un sous-groupe libre d'indice fini. Le groupe $\Gamma$
contient-il un sous-groupe isomorphe au groupe fondamental d'une surface de genre au moins $2$ ?
\end{ques}

On renvoie \`a \cite{Calegari} pour des travaux r\'ecents en lien avec cette question.

\section{Construction d'une surface \`a partir de pantalons}

Dans toute cette partie, $M = \Gamma \backslash \H_3$ est une vari\'et\'e hyperbolique de dimension $3$ uniformis\'ee, o\`u $\Gamma$ est un sous-groupe discret et sans torsion du groupe $\PSL_2 (\C)$. De cette mani\`ere, on identifie le fibr\'e des rep\`eres de $M$ \`a $\Gamma \backslash \PSL_2 (\C)$.

\begin{defi}
Soit $S$ une surface dont on fixe un rev\^etement universel $\widetilde{S}\rightarrow S$, de groupe de rev\^etement $\pi_1 (S)$. Une structure de {\rm $\PSL_2 (\C)$-surface} sur $S$ est la donn\'ee d'une immersion $F : \widetilde{S} \rightarrow \H_3$ et d'une repr\'esentation $\rho \in \mathrm{Hom} (\pi_1 (S) , \PSL_2 (\C))$ telles que $F$ soit  \'equivariante relativement \`a $\rho$. Deux structures de $\PSL_2 (\C)$-surfaces $(F_1 , \rho_1)$ et $(F_2 , \rho_2)$ sur $S$ sont {\rm \'equivalentes} s'il existe un \'el\'ement $g$ de 
$\PSL_2 (\C)$ tel que $F_2 = g F_1$ et $\rho_2 = g \rho_1 g^{-1}$. 
\end{defi}

Par analogie avec la notion de $(G,X)$-structure, on dira que $F$ est l'{\it application d\'eveloppante} et $\rho$ le {\it morphisme d'holonomie} de la structure de $\PSL_2 (\C)$-surface. Le groupe de surface du 
th\'eor\`eme \ref{T1} sera obtenu comme groupe de monodromie d'une structure de $\PSL_2 (\C)$-surface sur une surface $S$ immerg\'ee dans $M$. On aura \'egalement besoin de consid\'erer des surfaces \`a bord; on demande alors que l'application d\'eveloppante envoie chaque composante de bord sur une g\'eod\'esique.

\subsection{L'espace des futals}

Fixons $\Pi$ un pantalon orient\'e (\`a bord). 
On num\'erote $C_1$, $C_2$ et $C_3$ les composantes de bord orient\'ees de $\Pi$
et on choisit un \'el\'ement $c_n \in \pi_1 (\Pi)$ dans la classe de
conjugaison correspondant \`a chaque composante de bord $C_n$ pour $n=1,2,3$, de sorte que 
$$c_1 c_2 c_3 = \mathrm{id}.$$

Rappelons qu'un \'el\'ement $A \in \PSL_2 (\C)$ -- vu comme isom\'etrie de $\H_3$ -- a une {\it longueur de translation complexe} $\ell (A) \in (\R_+^* + i \R) / 2i\pi \Z$, telle que\footnote{On prendra garde au fait que la trace de $A$ n'est d\'efinie qu'au signe pr\`es et que la demi-longueur $\ell (A)/2$ n'est d\'efinie que modulo $i\pi$.} $\mathrm{trace} (A) = \pm 2 \cosh ( \ell (A)/2)$. 

\'Etant donn\'e un morphisme $\rho : \pi_1 (\Pi ) \rightarrow \PSL_2 (\C)$, on note $\ell_n= \ell_n (\rho)$ les longueurs de translation complexes des $\rho (c_n)$ pour $n=1,2,3$.

Kourouniotis \cite[Prop. 1.6]{Kourouniotis} montre que trois nombres complexes $\sigma_n \in (\R_+^* + i \R) / 2i\pi \Z$ o\`u $n=1,2,3$, d\'eterminent deux hexagones gauches \`a angles droits isom\'etriques, mais d'orientations oppos\'ees, dont trois c\^ot\'es non-adjacents sont de longueurs complexes $\hat{\sigma}_1$, $\hat{\sigma}_2$, $\hat{\sigma}_3$ pour un choix de $\hat{\sigma}_n \in \{
\sigma_n , \sigma_n + i \pi \}$. En recollant ces deux hexagones on munit le pantalon $\Pi$ d'une structure de $\PSL_2 (\C)$-surface de repr\'esentation d'holonomie $\rho : \pi_1 (\Pi) \rightarrow \PSL_2 (\C)$ telle que $\ell_n (\rho) = 2 \sigma_{n}$ pour $n=1,2,3$.
R\'eciproquement, si $\rho$ est le morphisme d'holonomie d'une structure de $\PSL_2 (\C)$-surface sur $\Pi$ et si les extr\'emit\'es des axes de translation des $\rho (c_n )$ pour $n=1,2,3$ sont deux \`a deux distinctes alors $\rho$ est conjugu\'ee \`a une repr\'esentation d'holonomie associ\'ee comme ci-dessus \`a trois {\it demi-longueurs} $\sigma_n = \sigma_n (\rho ) \in (\R_+^* + i \R) / 2i\pi \Z$ o\`u $n=1,2,3$. 

\begin{rema} \label{rem:1}
1. Quitte \`a conjuguer $\rho$ dans $\PSL_2 (\C)$, on peut supposer que 
l'axe de translation de $\rho (c_1)$ est la g\'eod\'esique (orient\'ee) $i \R_+^*$ de $\H_3$ qui va de $0$ \`a l'infini. Les perpendiculaires communes \`a $i\R_+^*$ et aux axes de translation des $\rho (c_n)$ o\`u $n=2,3$, d\'eterminent deux \'el\'ements $v_-$ et $v_+$, de points bases $p_-$ et $p_+$, du fibr\'e unitaire normal \`a $i \R_+^*$ que l'on num\'erote de sorte que l'orientation du segment $[p_- , p_+]$ co\"{\i}ncide avec l'orientation de la g\'eod\'esique $i\R_+^*$. Alors l'homographie $z \mapsto e^{\sigma_1} z$ est l'unique isom\'etrie directe de $\H_3$ qui pr\'eserve l'axe $i\R_+^*$ et envoie $v_-$ sur $v_+$. Cela d\'etermine l'\'el\'ement $\sigma_1 \in (\R_+^* + i \R) / 2i\pi \Z$, on proc\`ede de m\^eme pour $\sigma_2$ et $\sigma_3$. 

2. Toute repr\'esentation $\rho$ comme ci-dessus admet un relev\'e $\tilde{\rho}$ \`a $\SL_2 (\C)$ tel que 
$\mathrm{trace} (\tilde{\rho}(c_n) ) = -2 \cosh \sigma_n$ pour $n=1,2,3$. Les $\sigma_n$ d\'eterminent la classe de $\SL_2 (\C)$-conjugaison de $\tilde{\rho}$.
\end{rema}

\begin{defi}
1. Un {\rm futal} dans $M$ est la classe de conjugaison $\mathbf{\Pi} = [\rho]$ dans 
$\Gamma$ d'un morphisme injectif $\rho : \pi_1 (\Pi^0 ) \rightarrow \Gamma \subset \PSL_2 (\C )$ tel que $\rho$ est l'holonomie d'une structure de $\PSL_2 (\C)$-surface sur $\Pi$, les extr\'emit\'es des axes de translation des $\rho (c_n )$ pour $n=1,2,3$ sont deux \`a deux distinctes et $\rho$ v\'erifie  
\begin{equation} \label{eq:1}
\sigma_n (\rho) = \frac{1}{2} \ell_n (\rho), \quad n=1, 2,3.
\end{equation}
Dans la suite on choisira toujours des repr\'esentants $\ell_n \in \R_+^* + i \R$ tels que $-\pi <  \mathrm{Im} (\ell_n ) \leq \pi$. 

2. Un {\rm futal marqu\'e} dans $M$ est la donn\'ee $(\mathbf{\Pi} , g^*)$ d'un futal et d'une g\'eod\'esique ferm\'ee orient\'ee $g^*$ dans $M$ qui repr\'esente la classe
de conjugaison de $\rho (c_n)$ pour un certain $n=1,2,3$. 
\end{defi}

On note $\mathcal{P}$ l'ensemble des futals marqu\'es dans $M$ muni de la topologie discr\`ete. Le groupe $\Gamma$ op\`ere, \`a gauche diagonalement et par conjugaison, sur $\Gamma^3$ et un futal marqu\'e $(\mathbf{\Pi} , g^*)$ est uniquement d\'etermin\'e par la classe de conjugaison du triplet ordonn\'e $(\gamma_1 , \gamma_2 , \gamma_3 ) = ( \rho (c_n) , \rho (c_{n+1}) , \rho (c_{n+2})) \in \Gamma^3$, o\`u l'entier $n$, consid\'er\'e modulo $3$, est tel que $g^*$ repr\'esente la classe de conjugaison de $\rho (c_n)$. Un tel triplet v\'erifie en outre la relation 
\begin{equation} \label{eq:21}
\gamma_1 \gamma_2 \gamma_3 = \mathrm{id}
\end{equation}
qui est pr\'eserv\'ee par l'action par conjugaison de $\Gamma$. Notons 
$$\widehat{\mathcal{P}} = \Gamma \backslash \left\{ (\gamma_1 , \gamma_2 , \gamma_3 ) \in \Gamma^3 \; : \; \gamma_1 \gamma_2 \gamma_3 = \mathrm{id} \right\}.$$
On identifiera dor\'enavant $\mathcal{P}$ \`a un sous-ensemble de $\widehat{\mathcal{P}}$.

L'ensemble $\widehat{\mathcal{P}}$ est naturellement muni d'une {\it involution} $\mathcal{R}$ d\'efinie par:
$$\mathcal{R} (\gamma_1,  \gamma_2 , \gamma_3) = (\gamma_1^{-1} , \gamma_3^{-1} , \gamma_2^{-1} )$$
et d'un {\it tricycle} $\mathrm{rot}$ d\'efinie par
$$\mathrm{rot} (\gamma_1 , \gamma_2 , \gamma_3) = (\gamma_2 , \gamma_3 , \gamma_1 ).$$

\begin{rema} 
On peut penser \`a l'involution $\mathcal{R}$ comme \`a l'application qui envoie un futal marqu\'e $(\mathbf{\Pi} , g^* )\in \mathcal{P}$ sur le futal $\mathbf{\Pi}$ muni de l'\og orientation oppos\'ee \fg \ et marqu\'e par la g\'eod\'esique $-g^*$ (c'est-\`a-dire $g^*$ munie de l'orientation oppos\'ee). Le tricycle $\mathrm{rot}$ pr\'eserve \'egalement le sous-ensemble $\mathcal{P}$; il correspond \`a un changement cyclique
du marquage.
\end{rema}

Soit $(\mathbf{\Pi} , g^*) = [\gamma_1 , \gamma_2 , \gamma_3] \in \mathcal{P}$ un futal marqu\'e. On note 
$\sigma = \sigma (\mathbf{\Pi} , g^*) \in (\R_+^* + i \R ) / 2i\pi \Z$ 
la demi-longueur complexe de la composante de bord marqu\'ee~: $\sigma = \frac{1}{2} \ell (\gamma_1 )$.

\subsection{\'Etiquetage des futals et construction d'une surface}

Soit $\mathcal{E}$ un ensemble fini.

\begin{defi}
Un {\rm \'etiquetage l\'egal} est un triplet $(E , \mathcal{R}_E , \mathrm{rot}_E)$, que par abus nous noterons juste $E$, constitu\'e de trois applications $E:\mathcal{E} \rightarrow \mathcal{P}$, $\mathcal{R}_E : \mathcal{E} \rightarrow \mathcal{E}$ et 
$\mathrm{rot}_E : \mathcal{E} \rightarrow \mathcal{E}$ telles que 
\begin{enumerate}
\item l'application $\mathcal{R}_E$ est involutive et v\'erifie  
$$\mathcal{R} \circ E = E \circ \mathcal{R}_E,$$
\item l'application $\mathrm{rot}_E$ est d'ordre $3$ et v\'erifie
$$\mathrm{rot} \circ E = E \circ \mathrm{rot}_{E}.$$
\end{enumerate}
\end{defi}

\'Etant donn\'e un \'etiquetage 
l\'egal $E$ et un \'el\'ement $e \in \mathcal{E}$, on note $(\mathbf{\Pi}_e , g_e^*)$ le futal marqu\'e $E(e)$.

\begin{defi}
Une involution $\tau : \mathcal{E} \rightarrow \mathcal{E}$ est dite {\rm admissible}, relativement \`a un \'etiquetage 
l\'egal $E$, si, pour tout $e$ dans $\mathcal{E}$, on a~:
$$g_{\tau (e)}^* =  - g_e^* .$$
\end{defi}

On associe \`a un \'etiquetage l\'egal $E:\mathcal{E} \rightarrow \mathcal{P}$ et \`a une involution admissible $\tau : \mathcal{E} \rightarrow \mathcal{E}$ un graphe $\mathcal{G}$ trivalent: l'ensemble des sommets est
$\mathcal{E} / \langle \mathrm{rot}_E \rangle$ et deux sommets sont reli\'es par une ar\^ete si et seulement s'ils ont des repr\'esentants dans 
$\mathcal{E}$ \'echang\'es par $\tau$. Quitte \`a ne consid\'erer qu'une composante connexe choisie arbitrairement, nous supposerons dor\'enavant que $\mathcal{G}$ est connexe. Le bord d'un \'epaississement de $\mathcal{G}$ est alors une surface $S$ qui fibre en cercles au-dessus de $\mathcal{G}$ et les cercles au-dessus des milieux des ar\^etes d\'ecoupent $S$ en pantalons orient\'es. On note $\Pi_e$ ($e \in \mathcal{E}$) le pantalon correspondant au sommet $[e]$ dans $\mathcal{E} / \langle \mathrm{rot}_E \rangle$ dont on marque la composante de bord associ\'ee
au milieu de l'ar\^ete joignant les sommets $[e]$ et $[\tau (e)]$. 
L'\'etiquetage l\'egal $E$ associe \`a chaque \'el\'ement $e\in \mathcal{E}$ un futal $\mathbf{\Pi}_e=[\rho_e]$, o\`u $\rho_e$ est le morphisme d'holonomie d'une structure de $\PSL_2 (\C)$-surface $(F_e, \rho_e)$ sur $\Pi_e$, ainsi qu'une g\'eod\'esique ferm\'ee orient\'ee $g_e^* \subset M$ correspondant \`a la composante de bord marqu\'ee de $\Pi_e$. Soit $\sigma_e = \sigma (\mathbf{\Pi}_e , g_e^*) \in  (\R_+^* + i \R) / 2i\pi \Z$. Puisque $\tau$ est admissible, on a~:
\begin{equation} \label{sigmatau}
\sigma_{\tau(e)} = \sigma_{e}.
\end{equation}
La surface \`a bord obtenue en recollant $\Pi_e$ et $\Pi_{\tau (e)}$ selon les composantes de bord marqu\'ees est donc naturellement munie d'une structure de $\PSL_2 (\C)$-surface qui induit sur $\Pi_e$ et $\Pi_{\tau (e)}$ des structures de $\PSL_2 (\C)$-surfaces \'equivalentes \`a $(F_e , \rho_e)$ et $(F_{\tau(e)} , \rho_{\tau(e)})$.\footnote{Noter qu'il est n\'ecessaire pour cela que $\sigma_{\tau(e)} = \sigma_{e}$ plut\^ot que $\sigma_{\tau(e)} = \sigma_{e} +i \pi$.} De cette mani\`ere on obtient la proposition suivante.

\begin{prop}
Pour tout \'etiquetage l\'egal $E:\mathcal{E} \rightarrow \mathcal{P}$ et toute involution admissible $\tau : \mathcal{E} \rightarrow \mathcal{E}$, il existe une surface $S$ munie d'une structure de $\PSL_2 (\C)$-surface de repr\'esentation d'holonomie
$$\rho_{E, \tau} : \pi_1 (S ) \rightarrow \Gamma$$
telle que $S$ est d\'ecoup\'ee en pantalons orient\'es $\Pi_e$ avec une composante de bord marqu\'ee et la structure de $\PSL_2 (\C)$-surface sur $S$ induit sur chaque $\Pi_e$ une structure de $\PSL_2 (\C)$-surface \'equivalente \`a $(F_e , \rho_e)$.
\end{prop}

\subsection{Coordonn\'ees de Fenchel-Nielsen} \label{par:FN}

Soit $\gamma$ un \'el\'ement de $\Gamma$ et $g$ la g\'eod\'esique ferm\'ee {\it non orient\'ee} dans $M$ associ\'ee \`a la classe de conjugaison de $\gamma$ dans $\Gamma$. On note $Z_{\gamma}$ le centralisateur de $\gamma$ dans $\PSL_2 (\C)$. La longueur complexe $\ell$ de $\gamma$, vue 
dans $\R_+^* +i \R$ avec $-\pi < \ell \leq \pi$, ne d\'epend que de $g$. On note $\widehat{\T}_{\gamma} = \langle \gamma \rangle \backslash Z_{\gamma}$ et $\T_g$ le tore euclidien $\C / \left( \sigma \Z + 2i\pi \Z \right)$ o\`u $\sigma= \frac12 \ell$. On note enfin $(\T_{M} , d)$
l'espace m\'etrique obtenu en formant la r\'eunion disjointe des $\T_{g}$, o\`u $g$ parcourt l'ensemble des g\'eod\'esiques ferm\'ees (non orient\'ees) dans $M$.

Si $h$ est un \'el\'ement de $\PSL_2 (\C)$ la conjugaison par $h$ induit un diff\'eomorphisme de $\widehat{\T}_{\gamma}$ vers $\widehat{\T}_{h \gamma h^{-1}}$. Mais si 
l'axe (orient\'e) de $\gamma$ est $i \R_+^*$, alors le groupe $Z_{\gamma}$ est \'egal \`a $\big\{ A_{\zeta} = \big( \begin{smallmatrix} e^{\zeta /2} & \\ & e^{-\zeta /2} \end{smallmatrix} 
\big) \; : \; \zeta \in \C \big\}$. L'isomorphisme $A_{\zeta} \mapsto \zeta$ identifie canoniquement
$Z_{\gamma}$ au groupe $\C$; il permet de r\'ealiser 
$\widehat{\T}_{\gamma}$ comme rev\^etement de degr\'e $2$ de $\T_g$.

Consid\'erons maintenant un futal marqu\'e $(\mathbf{\Pi} , g^*) = [\gamma_1 , \gamma_2 , \gamma_3] \in \mathcal{P}$. Comme dans la remarque~\ref{rem:1}~(1), la perpendiculaire commune aux axes de $\gamma_1$ et $\gamma_2$, resp. $\gamma_1$ et $\gamma_3$, d\'etermine un 
\'el\'ement du fibr\'e unitaire normal \`a l'axe de $\gamma_1$ et donc, en prenant d'abord le vecteur tangent de l'axe de translation orient\'e de $\gamma_1$ puis en compl\'etant par leur produit vectoriel, un \'el\'ement du fibr\'e des rep\`eres de $\H_3$. On note $r_-$ et $r_+$ ces deux rep\`eres, 
ordonn\'es de sorte que l'on passe de $r_-$ \`a $r_+$ en suivant l'axe de $\gamma_1$ dans le sens positif. Le groupe $Z_{\gamma_1}$ op\`ere simplement transitivement sur le sous-ensemble constitu\'e des rep\`eres en un point de l'axe de $\gamma_1$. On \'ecrit  $r_- = A_- \cdot (p_0, \vec{u}_0 , \vec{n}_0 , \vec{u}_0 \wedge \vec{n}_0 )$ et $r_+ = A_+ \cdot (p_0, \vec{u}_0 , \vec{n}_0 , \vec{u}_0 \wedge \vec{n}_0 )$ avec $A_-$ et $A_+$ dans $Z_{\gamma_1}$. Il d\'ecoule de la remarque~\ref{rem:1} que $A_+ = A_{\sigma_1} A_-$; dans la suite on identifie les tore euclidiens $\T_g$ et $\langle A_{\sigma_1} \rangle \backslash Z_{\gamma_1}$. On prendra cependant garde au fait que l'action $(\zeta , B) \mapsto A_{\zeta} B$ de $\C$ sur $\langle A_{\sigma_1} \rangle \backslash Z_{\gamma_1}$ d\'epend d'un choix d'orientation de $g$.

\begin{defi}
On appelle {\rm pied} du futal marqu\'e $(\mathbf{\Pi} , g^*)$ l'image commune $p ( \mathbf{\Pi} , g^*)$
de $A_-$ et $A_+$ dans le tore $\T_g$. On note $p:
\mathcal{P} \rightarrow \T_{M}$ l'application qui \`a un futal marqu\'e $(\mathbf{\Pi} , g^*)$ associe son pied $p ( \mathbf{\Pi} , g^*) \in \T_g \subset \T_{M}$.
\end{defi}

Consid\'erons maintenant un \'etiquetage l\'egal $E:\mathcal{E} \rightarrow \mathcal{P}$ et une involution admissible $\tau : \mathcal{E} \rightarrow \mathcal{E}$.
Soit $e\in \mathcal{E}$. On a $\sigma_{\tau(e)} = \sigma_{e}$
mais en g\'en\'eral $p( \mathbf{\Pi}_e , g_e^* ) \neq p (\mathbf{\Pi}_{\tau(e)} , g_{\tau(e)}^*)$. 

\begin{defi}  
On appelle {\rm param\`etre de d\'ecalage} associ\'e \`a la g\'eod\'esique marqu\'ee $g_e^*$ l'\'el\'ement $t_e \in \C/ \left( \sigma_e \Z + 2i\pi \Z \right)$ tel que
$$A_{t_e+i\pi}  p(   \mathbf{\Pi}_e , g_e^* ) =  p (\mathbf{\Pi}_{\tau(e)} , g_{\tau(e)}^*).$$
\end{defi}

\begin{rema}
L'action $(\zeta , B) \mapsto A_{\zeta} B$ de $\C$ sur $\T_{g_e}$ d\'epend du choix d'orientation $g_e^*$, le changer revient \`a changer l'action en 
$(\zeta , B) \mapsto A_{\zeta}^{-1} B$. On en d\'eduit que 
$t_e = t_{\tau(e)}$ et donc que, comme la demi-longueur $\sigma_e$, le param\`etre de d\'ecalage $t_e$ ne d\'epend que de la composante de bord marqu\'ee et pas du futal.
\end{rema}

Les demi-longueurs complexes $(\sigma_e )_{e \in \mathcal{E}}$ d\'eterminent une structure de $\PSL_2 (\C)$-surface sur chaque pantalon $\Pi_e$. Si tous les param\`etres $(\sigma_e , t_e )_{e \in \mathcal{E}}$ sont {\it r\'eels} la repr\'esentation $\rho_{E, \tau}$ est injective et
conjugu\'ee \`a une repr\'esentation dans $\PSL_2 (\R)$. Elle est d\'etermin\'ee, \`a conjugaison pr\`es, par les param\`etres (de Fenchel-Nielsen) $(\sigma_e , t_e/ \sigma_e)_{e \in \mathcal{E}} \in (\R_+^{\times})^{\mathcal{E}} \times \left( \R / \Z\right)^{\mathcal{E}}$. Cette derni\`ere assertion reste vraie lorsque les param\`etres sont complexes~: la repr\'esentation $\rho_{E, \tau} : \pi_1 (S) \rightarrow \PSL_2 (\C)$ est encore 
d\'etermin\'ee, \`a conjugaison pr\`es, par les param\`etres $(\sigma_e , t_e )_{e \in \mathcal{E}}$ qui sont les param\`etres de Fenchel-Nielsen complexes introduits par Kourouniotis \cite{Kourouniotis} et Tan~\cite{Tan}.  

\subsection{Surfaces presque plates}

\begin{defi} 
Soient $R$ et $\varepsilon$ deux r\'eels strictement positifs. Un futal $\mathbf{\Pi} = [\rho]$ dans $M$ est {\em $(R, \varepsilon )$-plat} si, pour $n=1,2,3$, on a:
$$\Big|\sigma_n (\rho) - \frac{R}{2} \Big| \leq \varepsilon.$$
\end{defi}

Dans la partie suivante, nous d\'eduirons le th\'eor\`eme \ref{T1} des deux th\'eor\`emes qui suivent. Le premier -- qui d\'ecoulera du fait que le flot des rep\`eres sur $M$ est exponentiellement m\'elangeant -- 
affirme que si l'on fixe $\varepsilon$ et laisse tendre $R$ vers l'infini, il existe \og beaucoup \fg \ de futals $(R, \varepsilon)$-plats dans $M$ et que  
les pieds de ces futals sont \og bien distribu\'es \fg. On dispose donc d'une grande flexibilit\'e pour construire des surfaces \`a partir de ces futals presque plats. Le
second th\'eor\`eme fournit une recette d'assemblage de futals presque plats de sorte que la repr\'esentation de monodromie de la structure de $\PSL_2 (\C)$-surface obtenue \`a l'issue de l'assemblage soit injective. 

\'Etant donn\'e un espace m\'etrique $(X, d)$, on note $\mathcal{M} (X)$ l'ensemble des mesures bor\'eliennes finies positives \`a support compact dans $X$. Par analogie avec la
m\'etrique de Levy-Prokhorov, voir \cite{Billingsley}, on pose:

\begin{defi} Deux mesures $\mu$, $\nu \in \mathcal{M} (X)$ sont dites {\rm $\delta$-proches}, pour un certain r\'eel strictement positif $\delta$, si 
\begin{enumerate}
\item $\mu (X) = \nu (X)$, et
\item pour tout bor\'elien $A$ de $X$, on a $\mu (A) \leq \nu (V_{\delta} (A))$, o\`u $V_{\delta} (A)$ d\'esigne le $\delta$-voisinage de $A$ dans $X$.
\end{enumerate}
\end{defi}
Noter qu'\^etre $\delta$-proche est une relation sym\'etrique.\footnote{En effet~: 
$\nu (A) \leq \nu (X) - \nu (V_{\delta} (X \setminus V_{\delta} (A)))  \leq \mu (X) - \mu (X \setminus V_{\delta} (A)) = \mu (V_{\delta} (A))$.}

Par abus de notation, on note $\mathcal{M} (\mathcal{P})$ l'espace des mesures bor\'eliennes positives $\mu$ \`a support fini sur $\mathcal{P}$ qui sont pr\'eserv\'ees \`a la fois par 
l'involution $\mathcal{R}$ et le tricycle $\mathrm{rot}$.
L'application $p$ qui \`a un futal marqu\'e $(\mathbf{\Pi} , g^*)$ associe son pied dans $\T_{g}$ permet de pousser une mesure $\mu \in \mathcal{M} (\mathcal{P})$
sur une mesure $p_* \mu \in \mathcal{M} (\T_{M})$. 

\begin{defi}
Soient $\delta$ un r\'eel strictement positif et $\nu$ une mesure dans $\mathcal{M} (\T_{M})$. 
On note $\Lambda_{\nu}$ la moyenne de $\nu$ sous l'action de $\C$, de sorte que~:
$$\Lambda_{\nu} \- {}_{| \T_g} = \left[ \frac{\nu (\T_g )}{\Lambda (\T_g )} \right] \Lambda,$$
o\`u $\Lambda$ est une mesure de Lebesgue sur $\T_g$. On dit que $\nu$ est {\rm $\delta$-\'equidistribu\'ee} si 
$\nu$ est $\delta$-proche de $\Lambda_{\nu}$. On \'ecrira alors $\nu \in \mathcal{M}_{\delta} (\T_M )$.
\end{defi}

\begin{theo} \label{TA}
Il existe des constantes strictement positives $q$ et $D$ telles que pour tout r\'eel $\varepsilon \in \; ]0,1]$, il existe un r\'eel $R_{\varepsilon}$ tel que pour tout 
$R \geq R_{\varepsilon}$, il existe une mesure non nulle $\mu= \mu_{\varepsilon , R} \in \mathcal{M} (\mathcal{P})$ telle que
\begin{enumerate}
\item $\mu$ est support\'ee sur les futals $(R , \varepsilon)$-plats, et 
\item $p_* \mu$ est $DRe^{-qR}$-\'equidistribu\'ee.
\end{enumerate}
\end{theo}

\begin{defi}
Soient $E:\mathcal{E} \rightarrow \mathcal{P}$ un \'etiquetage l\'egal, $\tau : \mathcal{E} \rightarrow \mathcal{E}$ une involution admissible et $\varepsilon$ et $R$ deux 
r\'eels strictement positifs. On dit que la repr\'esentation
 $\rho_{E, \tau}$ est {\rm $(R , \varepsilon)$-plate} si pour tout $e \in \mathcal{E}$,
 \begin{enumerate}
 \item le futal $\mathbf{\Pi}_e$ est $(R , \varepsilon)$-plat, ou encore $|\sigma_e - R/2| \leq \varepsilon$, et 
 \item le param\`etre de d\'ecalage $t_e$ est $\frac{\varepsilon}{R}$-proche de $1$~: $| t_e - 1 | \leq \varepsilon /R$.
\end{enumerate}
\end{defi} 

Si $R$ est fix\'e alors pour $\varepsilon$ suffisamment petit, une repr\'esentation $(R, \varepsilon)$-plate est la monodromie d'une structure de $\PSL_2 (\C)$-surface qui est arbitrairement proche d'\^etre fuchsienne ; en particulier elle est injective. L'int\'er\^et du r\'esultat qui suit est de montrer que puisque l'on a impos\'e aux param\`etres de d\'ecalage d'\^etre proches de $1$, on peut en fait choisir $\varepsilon$ uniforme par rapport \`a $R$.

\begin{theo} \label{TB}
Il existe des r\'eels strictement positifs $\varepsilon_0$ et $R_0$ tels que pour tout \'etiquetage l\'egal $E:\mathcal{E} \rightarrow \mathcal{P}$ et toute involution admissible $\tau : \mathcal{E} \rightarrow \mathcal{E}$ et quel que soit $R \geq R_0$, $\varepsilon \in \; ]0,  \varepsilon_0]$, si la repr\'esentation $\rho_{E, \tau}$ est $(R , \varepsilon)$-plate alors
$\rho_{E, \tau}$ est injective.
\end{theo}

\section{D\'emonstration du th\'eor\`eme \ref{T1}}

Dans cette partie, on admet les th\'eor\`emes \ref{TA} et \ref{TB} et on en d\'eduit le th\'eor\`eme~\ref{T1}. Soient $q$, $D$ et $\varepsilon_0$ les constantes fournies par les th\'eor\`emes \ref{TA} et \ref{TB}. On fixe $\varepsilon \in \; ]0 , \min (\varepsilon_0 , 1) ]$. Soit $R_{\varepsilon}$ le nombre r\'eel strictement positif qui lui est associ\'e par le th\'eor\`eme \ref{TA}. La proposition qui suit garantit l'existence d'une repr\'esentation $(R, \varepsilon)$-plate avec $R\geq R_0$ et $\varepsilon \leq \varepsilon_0$. Le th\'eor\`eme \ref{TB} implique alors que cette repr\'esentation est injective, ce qui d\'emontre le th\'eor\`eme \ref{T1}. 

\begin{prop} \label{PA}
Soit $R > R_{\varepsilon}$. Il existe un \'etiquetage l\'egal $E:\mathcal{E} \rightarrow \mathcal{P}$ et une involution admissible $\tau : \mathcal{E} \rightarrow \mathcal{E}$
tels que pour tout $e \in \mathcal{E}$, on a~:
\begin{enumerate}
\item $|\sigma_e - R/2| \leq \varepsilon$, et 
\item $|t_e -1 | \leq 2D R e^{-q R} $.  
\end{enumerate}
\end{prop}

On explique maintenant comment d\'eduire la proposition \ref{PA} du th\'eor\`eme \ref{TA}. Il s'agit de passer des mesures aux \'etiquetages. On fixe $R>R_{\varepsilon}$.

\subsection{Une cons\'equence m\'etrique du th\'eor\`eme des mariages de Hall}

Une mesure $\mu \in \mathcal{M} (\mathcal{P})$ est presque un \'etiquetage, la difficult\'e est de construire l'involution admissible $\tau$. L'outil technique est la proposition
g\'en\'erale suivante. 

\begin{prop} \label{P:Hall}
Soient $(X,d)$ un espace m\'etrique, $A$ et $B$ deux ensembles finis de m\^eme cardinal, $f : A \rightarrow X$ et $g: B \rightarrow X$ deux applications et $\delta$ un r\'eel strictement positif. On note $\#_A$, resp. $\#_B$, la mesure de comptage sur $A$, resp. $B$. Et on suppose que les mesures $f_* \#_A$ et  $g_* \#_B$ sont $\delta$-proches.  Alors, il existe une bijection $h : A \rightarrow B$ telle que pour tout $a \in A$, on ait $d(g(h(a)), f(a)) \leq \delta$.
\end{prop}
\begin{proof}
On d\'efinit~:
$$M = \{ (a,b) \in A \times B \; : \; d(f(a) , g(b)) \leq \delta \}$$
comme ensemble des mariages possibles. Il s'agit de construire une bijection $h : A \rightarrow B$ telle que pour tout $a \in A$, on ait $(a, h(a)) \in M$. Le th\'eor\`eme des mariages de Philip Hall \cite{Hall} fournit un crit\`ere pour l'existence
d'une telle {\it application de mariage}~:

\begin{lemm} \label{L:Hall}
Pour qu'il existe une application de mariage, il suffit que pour tout sous-ensemble $C$ de $A$ il y ait au moins $|C|$ \'el\'ements de $B$ qui puissent \^etre mari\'es \`a 
un \'el\'ement de $C$. 
\end{lemm}

Il s'agit donc de montrer que si $C$ est un sous-ensemble de $A$, le sous-ensemble $M_C = \{ b \in B \; : \; \exists \; a \in C, \ (a,b) \in M \}$ de $B$ contient au moins $|C|$ \'el\'ements.
Cela r\'esulte du calcul suivant, reposant sur l'hypoth\`ese que les mesures $f_* \#_A$ et  $g_* \#_B$ sont $\delta$-proches~:
\begin{equation*}
|M_C | = (g_* \#_B) (V_{\delta} (f( C))) \geq (f_* \#_A ) (f( C)) \geq |C|.
\end{equation*}~\end{proof}

\subsection{D\'emonstration de la proposition \ref{PA}}
Le th\'eor\`eme \ref{TA} fournit une mesure $\mu \in \mathcal{M} (\mathcal{P})$ telle que 
$p_* \mu$ soit $DRe^{-qR}$-proche de sa moyenne sous l'action de $\C$. En particulier, on a~:
\begin{equation} \label{RL}
p_* \mu  \mbox{ et } (A_{1+i\pi} \circ p )_* \mu   \mbox{ sont } 2DRe^{-qR}\mbox{-proches} .
\end{equation}
On commence par montrer qu'il existe une mesure {\it rationnelle} v\'erifiant \eqref{RL}.

\begin{lemm}
Il existe une mesure non nulle $\mu^{\rm rat} \in \mathcal{M} (\mathcal{P})$ support\'ee par des futals $(R, \varepsilon)$-plats, v\'erifiant \eqref{RL} et telle que pour tout $(\mathbf{\Pi} , g^*) 
\in \mathcal{P}$, on ait $\mu^{\rm rat} (\mathbf{\Pi} , g^*) \in \Q$.
\end{lemm}
\begin{proof}
Les mesures $p_* \mu$ et $(A_{1+i\pi} \circ p )_* \mu$ sont atomiques \`a support fini. Notons $a_i$ et $b_j$ leurs atomes respectifs. La relation \eqref{RL} est alors \'equivalente \`a 
un syst\`eme  lin\'eaire fini $S$ d'in\'egalit\'es entre les valeurs $x_i = p_* \mu (a_i )$ et $y_j = (A_{1+i\pi} \circ p )_* \mu (b_j )$ dont les coefficients sont entiers (et m\^eme
dans $\{0,1 \}$). Puisque le syst\`eme $S$ a une solution r\'eelle, on peut trouver des solutions rationnelles $x_i^{\rm rat}$, $y_j^{\rm rat}$ arbitrairement proches. On peut
alors remplacer $\mu$ par une mesure $\mu^{\rm rat}$ avec les m\^emes atomes mais rationnelle et telle que $x_i^{\rm rat} = p_* \mu^{\rm rat} (a_i )$ et $y_j^{\rm rat} = (A_{1+i\pi} \circ p )_* \mu^{\rm rat} (b_j )$.
La mesure $\mu^{\rm rat}$ v\'erifie les conclusions du lemme.
\end{proof}

Nous supposons dor\'enavant que $\mu = \mu^{\rm rat}$. Quitte \`a multiplier $\mu$ par un entier suffisamment grand, on peut m\^eme supposer que chaque  $\mu (\mathbf{\Pi} , g^*)$ est entier.
On peut alors \'ecrire $\mu$ comme une somme formelle $\mathcal{R}$-sym\'etrique de futals (non marqu\'es)\footnote{La mesure $\mu$, \'etant $\mathrm{rot}$-invariante,  attribue le m\^eme poids \`a deux futals marqu\'es qui ne diff\`erent que par le marquage.}~:
$$\mu = n_1 (\mathbf{\Pi}_1 + \mathcal{R} (\mathbf{\Pi}_1)) + n_2 (\mathbf{\Pi}_2 + \mathcal{R} (\mathbf{\Pi}_2)) + \ldots + n_m (\mathbf{\Pi}_m + \mathcal{R} (\mathbf{\Pi}_m)), \quad (n_i \in \N^*).$$
Pour chaque $s=1 , \ldots , m$, on fixe un marquage $(\mathbf{\Pi}_s , g_s^*)$. Consid\'erons maintenant l'ensemble $\mathcal{E} = \{ (j,k) \; : \; j=1,2, \ldots , 2(n_1 +n_2 + \ldots +n_m), \ k\in \Z/3 \Z\}$. On associe \`a $\mu$ l'\'etiquetage $E : \mathcal{E} \rightarrow \mathcal{P}$ d\'efini par
$$E(j,k) = \left\{ \begin{array}{l}
\mathrm{rot}^k (\mathbf{\Pi}_s , g_s^*) \quad \mbox{ si } j\mbox{ est impair et } 2(n_1+ \ldots + n_{s-1}) < j \leq 2 (n_1+ \ldots + n_s) \\
\mathcal{R} \circ \mathrm{rot}^k (\mathbf{\Pi}_s , g_s^*) \ \mbox{ si } j\mbox{ est pair et } 2(n_1+ \ldots + n_{s-1}) < j \leq 2 (n_1+ \ldots + n_s).
\end{array} \right.$$
L'ensemble $\mathcal{E}$ est naturellement muni d'une involution $\mathcal{R}_E: (j,k) \mapsto (j+(-1)^{j+1} , k)$ et d'un tricycle $\mathrm{rot}_E : (j,k) \mapsto (j, k+1)$ qui font
de l'application $E$ un \'etiquetage l\'egal. Noter que, puisque $\mu$ est support\'ee par des futals $(R, \varepsilon)$-plats, pour tout $e \in \mathcal{E}$, on a $|\sigma_e -R/2| \leq \varepsilon$.

Il nous reste \`a construire une involution admissible $\tau : \mathcal{E} \rightarrow \mathcal{E}$ telle que, pour tout $e \in \mathcal{E}$, on ait 
$d ( A_{1+i\pi}  p (\mathbf{\Pi}_e , g_e^*) , p (\Pi_{\tau (e)}, g_{\tau (e)}^* ) ) \leq 2D R e^{-q R} $.  Pour cela, on d\'ecompose $\mathcal{E}$ en la r\'eunion disjointe des 
sous-ensembles
$$\mathcal{E} (g^* )^{\pm} = \{ (j,k ) \; : \; E(j,k) = (\cdot , g^* ), \ \pm = (-1)^{j+1} \},$$
o\`u $g^*$ parcourt l'ensemble des g\'eod\'esiques ferm\'ees orient\'ees de $M$. 

\begin{lemm}
Soit $g^*$ une g\'eod\'esique ferm\'ee orient\'ee de $M$. Il existe une bijection $h=h_{g^*} :  \mathcal{E} (g^*)^+ \rightarrow  \mathcal{E} (-g^*)^-$ telle que pour tout $e \in \mathcal{E} (g^*)^+$, 
$$d \left( (A_{1+i\pi}  \circ p \circ E) (h(e)), (p\circ E) (e) \right) \leq 2D R e^{-q R}.$$
\end{lemm} 
\begin{proof} Posons $\alpha^+ = (p \circ E)_* \#_{ \mathcal{E} (g^*)^+}$ et $\alpha^- = (p \circ E)_* \#_{ \mathcal{E} (-g^*)^-}$.
Il d\'ecoule de la $\mathcal{R}$-sym\'etrie de $\mu$ (ou de l'\'etiquetage $E$) que les mesures $\alpha^+$ et $\alpha^-$ co\"{\i}ncident. On a alors 
$p_* \mu _{|\T_g } = 4 \alpha^- = 4 \alpha^+$ et l'\'equation \eqref{RL} implique que les mesures 
$\alpha^+$ et  $(A_{1+i\pi})_* \alpha^-$ sont  $2DRe^{-qR}$-proches. 
Le lemme d\'ecoule donc de la proposition \ref{P:Hall}.
\end{proof}

On d\'efinit alors l'involution $\tau : \mathcal{E} \rightarrow \mathcal{E}$ par 
$$\tau (e) = \left\{ \begin{array}{ll}
h_{g^*} (e) & \mbox{ si } e \in \mathcal{E} (g^*)^+ \\ 
h_{-g^*} ^{-1} (e) & \mbox{ si } e \in \mathcal{E} (g^*)^-.
\end{array} 
\right.$$
La proposition \ref{PA} est d\'emontr\'e et donc aussi le th\'eor\`eme \ref{T1}.

\section{D\'emonstration du th\'eor\`eme \ref{TA}}

On note $K$ le stabilisateur de $p_0$ dans $\PSL_2 (\C)$ et $a_t$ la matrice $\big( \begin{smallmatrix}
e^{t/2} & 0 \\ 0 & e^{-t/2} \end{smallmatrix} \big)$, o\`u $t \in \R$.

\subsection{Le flot des rep\`eres}

La courbe $t \mapsto p_t = (0,e^t) =a_t \cdot p_0$ est la g\'eod\'esique dans $\H_3$ de conditions initiales $(p_0 , \vec{u}_0)$. Le rep\`ere $(p_t , \vec{u}_t , \vec{n}_t , \vec{u}_t \wedge \vec{n}_t )$ associ\'e \`a la matrice $a_t$ est obtenu par transport parall\`ele de $(p_0 , \vec{u}_0 , \vec{n}_0 , \vec{u}_0 \wedge \vec{n}_0 )$ pendant un temps $t$ le long de la g\'eod\'esique $(p_t)_{t\in \R}$. Et pour tout $g \in \PSL_2 (\C)$, la g\'eod\'esique $(g\cdot p_t)_{t \in \R}$ a pour condition initiale $g \cdot (p_0 , \vec{u}_0 )$. L'action du flot g\'eod\'esique sur le fibr\'e des 
rep\`eres par transport parall\`ele pendant un temps $t$ correspond donc -- dans $\PSL_2 (\C)$ -- \`a la multiplication \`a droite par la matrice $a_t$. Cette action induit une action 
-- appel\'ee {\it flot des rep\`eres} -- sur $\Gamma \backslash \PSL_2 (\C)$. La mesure de Haar sur $\PSL_2 (\C )$ induit une mesure finie -- invariante
sous l'action du flot des rep\`eres -- sur $\Gamma \backslash \PSL_2 (\C)$; on note $\nu$ la mesure de probabilit\'e correspondante. 
Le th\'eor\`eme suivant -- d\^u \`a Calvin Moore \cite{Moore} -- affirme que le flot des rep\`eres est exponentiellement m\'elangeant relativement \`a la mesure de Haar; 
on donne les grandes id\'ees de sa d\'emonstration en appendice.

\begin{theo} \label{ME}
Il existe une constante $q=q(M)>0$ telle que si $f$ et $g$ sont deux fonctions de classe $C^{\infty}$ de $\Gamma \backslash \PSL_2 (\C )$ dans $\R$, alors il existe une constante 
$C$, fonction continue des normes de Sobolev de $f$ et de $g$, telle que pour tout $t \in \R$,
$$\Big| \int_{\Gamma \backslash \PSL_2 (\C)} g(x a_t ) f(x) d\nu (x) - \Big( \int_{\Gamma \backslash \PSL_2 (\C)} f d\nu \Big)  \Big( \int_{\Gamma \backslash \PSL_2 (\C)} g d\nu \Big) \Big|
\leq C e^{-q|t|}.$$
\end{theo}

\subsection{Une application du m\'elange}

On munit le groupe $\PSL_2 (\C)$ d'une m\'etrique riemannienne $d$ invariante \`a gauche sous l'action de $\PSL_2 (\C)$, invariante \`a droite sous l'action de $K$ et qui rel\`eve la m\'etrique hyperbolique sur $\H_3$. 
Soit $\varepsilon$ un r\'eel strictement positif. On fixe une fonction $\chi_{\varepsilon}$ positive, de classe $C^{\infty}$ sur $\PSL_2 (\C)$, support\'ee
dans la boule de rayon $\varepsilon/2$ centr\'ee en l'identit\'e et v\'erifiant:
\begin{enumerate}
\item $\chi_{\varepsilon} (\mathrm{Id}) >0$, et
\item $\int_{\PSL_2 (\C )} \chi_{\varepsilon} (x) d\nu(x) =1$. 
\end{enumerate}
On associe \`a la fonction $\chi_{\varepsilon}$ le noyau $k_{\varepsilon} (x,y) = \chi_{\varepsilon} (x^{-1} y)$ sur $\PSL_2 (\C )$. 

\begin{defi}
Soit $r$ un r\'eel strictement positif. On pose~:
$$\forall  \; x,y \in \PSL_2 (\C ), \quad a_{\varepsilon , r} (x,y) = \int_{\PSL_2 (\C )} k_{\varepsilon} (x, g a_r) k_{\varepsilon} (y,g) d\nu (g)$$
et 
$$\forall \; x,y \in \PSL_2 (\C ), \quad A_{\varepsilon , r} (x,y) = \sum_{\gamma \in \Gamma} a_{\varepsilon , r} (x,\gamma y).$$
\end{defi}
Le noyau $a_{\varepsilon, r}$ est un noyau {\it invariant} sur $\PSL_2 (\C)$~: $a_{\varepsilon , r}
(gx,gy) = a_{\varepsilon , r} (x,y)$, pout tous $x,y,g \in \PSL_2 (\C )$. Le noyau 
$A_{\varepsilon , r}$ d\'efinit lui un noyau $C^{\infty}$ sur $\Gamma \backslash \PSL_2 (\C )$; il mesure
si deux rep\`eres dans $M$ sont proches d'\^etre envoy\'es l'un sur l'autre par l'action du flot des rep\`eres pendant un temps $r$.

\begin{prop} \label{P:ME}
Il existe une constante 
$C=C(\varepsilon,M)$ telle que pour tous $x,y \in \PSL_2 (\C)$, on a~:
$$| A_{\varepsilon , r} (x,y) -1 | \leq C e^{-qr}.$$
\end{prop}
\begin{proof} Soient $x,y \in \PSL_2 (\C)$ et $D \subset \PSL_2 (\C)$ un domaine fondamental pour l'action de $\Gamma$. On a~:
\begin{equation*}
\begin{split}
A_{\varepsilon , r} (x,y) & = \sum_{\gamma \in \Gamma} \int_{\PSL_2 (\C)} k_{\varepsilon} (x, g a_r) k_{\varepsilon} (y , \gamma^{-1} g) d\nu(g) \\
& = \sum_{\gamma , \gamma ' \in \Gamma} \int_{D} k_{\varepsilon} (x, \gamma' g a_r ) k_{\varepsilon} (y , \gamma^{-1} \gamma ' g) d\nu(g) \\
& = \int_{\Gamma \backslash \PSL_2 (\C )} K_{\varepsilon} (x, ga_r ) K_{\varepsilon} (y , g) d\nu(g),
\end{split}
\end{equation*}
o\`u $K_{\varepsilon} (x,y) = \sum_{\gamma \in \Gamma} k_{\varepsilon} (x, \gamma y)$ d\'efinit une fonction de classe $C^{\infty}$ sur $\Gamma \backslash \PSL_2 (\C )\times \Gamma \backslash \PSL_2 (\C )$.
Le th\'eor\`eme \ref{ME}\footnote{Noter que l'on a~:
$$\int_{\Gamma \backslash \PSL_2 (\C )} K_{\varepsilon} (x, g) d\nu(g) = \int_{\PSL_2 (\C )} k_{\varepsilon} (x,  g) d\nu(g) =1.$$} appliqu\'e aux fonctions 
$f= K_{\varepsilon} (y , \cdot)$ et $g=K_{\varepsilon} (x , \cdot)$, fournit alors une constante $c(x,y)$ telle que  $| A_{\varepsilon , r} (x,y) -1 | \leq c(x,y) e^{-qr}$.
Enfin, par compacit\'e de $M$, on peut choisir $C$ majorant uniform\'ement $c(x,y)$.
\end{proof}

\begin{rema} 
Il d\'ecoule de la proposition \ref{P:ME} que pour $r$ suffisamment grand, il existe
$\gamma \in \Gamma$ tel que $a_{\varepsilon , r} (x, \gamma x)$ soit strictement positif. 
Il existe alors $g \in \PSL_2 (\C)$ tel que $d(g \gamma^{-1} g,  a_r) \leq \varepsilon$. 
Et la g\'eod\'esique ferm\'ee orient\'ee dans $M$ associ\'ee \`a $\gamma$ a donc pour longueur (complexe) un nombre $\ell$ tel que
$|\ell - r | < \varepsilon$. Ce r\'esultat -- ainsi que le principe de sa d\'emonstration -- est d\^u \`a Gregori Margulis \cite{Margulis}.
\end{rema}

\subsection{Des rep\`eres aux futals}

Soit $x \in \PSL_2 (\C)$; on note $(p, \vec{u} , \vec{n}, \vec{u} \wedge \vec{n}) = x\cdot (p_0 , \vec{u}_0 , \vec{n}_0 , \vec{u}_0 \wedge \vec{n}_0)$.
\'Etant donn\'e un r\'eel $r$, on associe \`a $x$ les \'el\'ements
$$x_{(k,r)} = x R_{\frac{2ik\pi}{3}} a_{r/4} \in \PSL_2 (\C)  \quad \mbox{pour } k =0, 1 , 2 ,$$
o\`u $R_{\theta} = \big( \begin{smallmatrix} \cos (\theta /2) & - i \sin (\theta /2) \\ - i \sin (\theta/2) & \cos (\theta /2) \end{smallmatrix} \big)$ est l'\'el\'ement de $\PSL_2 (\C)$ d'extension de Poincar\'e la rotation d'angle $\theta$ dans le 
sens direct autour de la g\'eod\'esique passant par $p_0$ et dirig\'ee par le vecteur $\vec{n}_0$.\footnote{On prendra garde au fait que $R_{\theta}$ et $a_t$ n'ont pas le m\^eme axe et ne commutent donc pas en g\'en\'eral.}  Soit $\omega^k (\vec{u})$ le vecteur obtenu
par rotation directe d'angle $\frac{2k\pi}{3}$ de $\vec{u}$ autour du vecteur $\vec{n}$. Le rep\`ere $x_{(k,r)} \cdot  (p_0 , \vec{u}_0 , \vec{n}_0 , \vec{u}_0 \wedge \vec{n}_0)$ est l'image de $(p, \omega^k 
(\vec{u}) , \vec{n} , \omega^k 
(\vec{u}) \wedge \vec{n})$ par le flot des rep\`eres pendant un temps~$r/4$. 

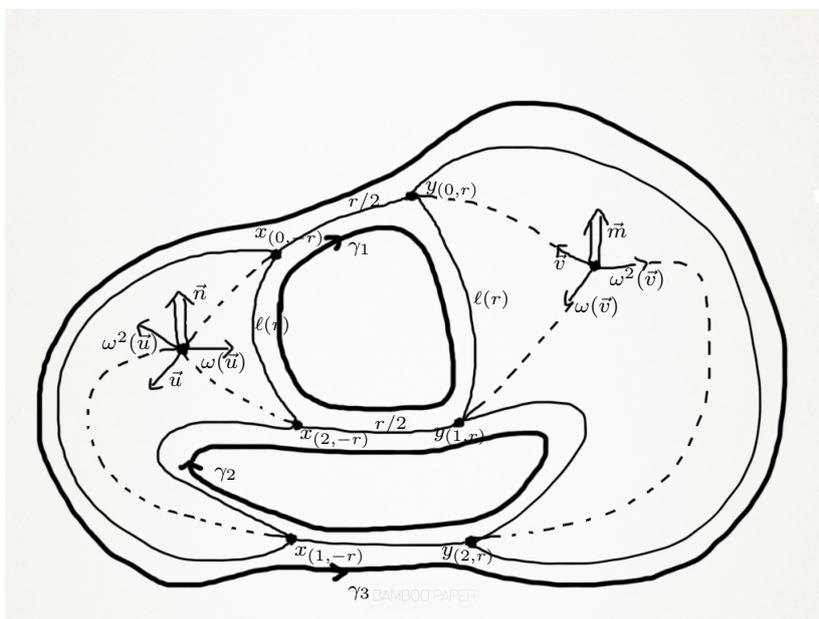
\begin{figure}[ht]      
\begin{center}
\includegraphics[scale=.3]{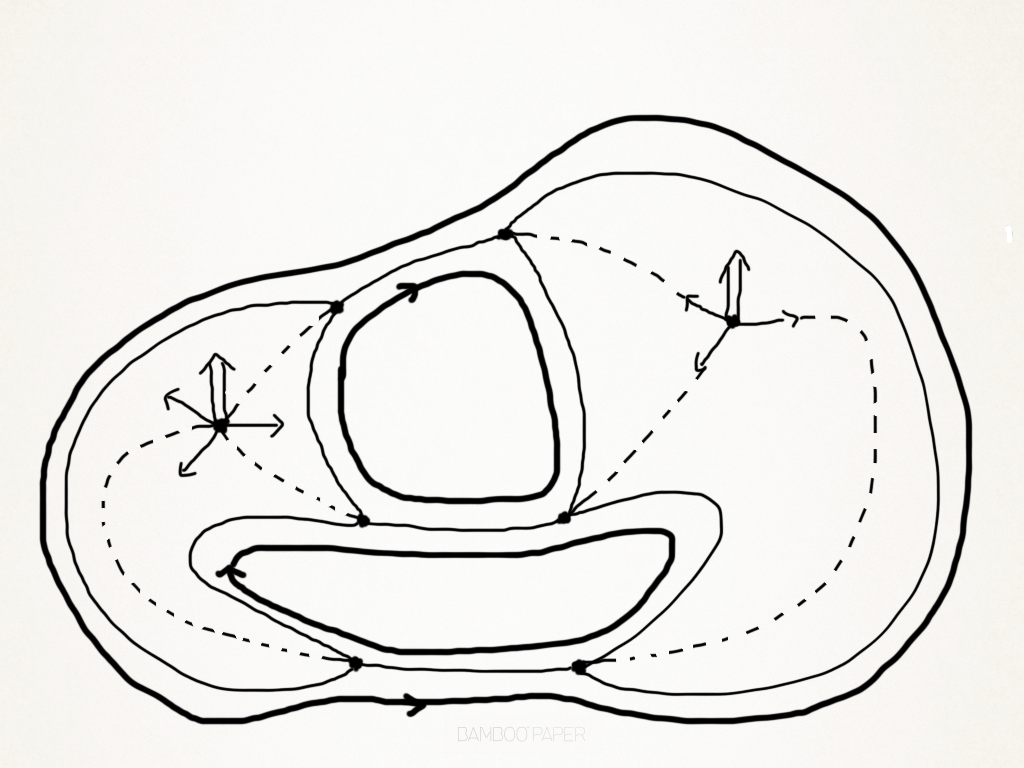}
\put(-180,10){$\scriptstyle{\gamma_3}$}
\put(-83,145){$\scriptstyle{\vec{m}}$}
\put(-82,127){$\scriptstyle{\omega^2 (\vec{v} )}$}
\put(-103,133){$\scriptstyle{\vec{v}}$}
\put(-95,117){$\scriptstyle{\omega (\vec{v} )}$}
\put(-145,25){$\scriptstyle{y_{(2,r)}}$}
\put(-180,140){$\scriptstyle{\gamma_{1}}$}
\put(-230,55){$\scriptstyle{\gamma_{2}}$}
\put(-235,96){$\scriptstyle{\omega (\vec{u} )}$}
\put(-247,89){$\scriptstyle{\vec{u}}$}
\put(-272,103){$\scriptstyle{\omega^2 (\vec{u} )}$}
\put(-238,122){$\scriptstyle{\vec{n}}$}
\put(-200,25){$\scriptstyle{x_{(1,-r)}}$}
\put(-198,69){$\scriptstyle{x_{(2,-r)}}$}
\put(-148,70){$\scriptstyle{y_{(1,r)}}$}
\put(-215,145){$\scriptstyle{x_{(0,-r)}}$}
\put(-151,163){$\scriptstyle{y_{(0,r)}}$}
\put(-133,120){\footnotesize{$\scriptstyle{\ell (r)}$}}
\put(-215,110){\footnotesize{$\scriptstyle{\ell (r)}$}}
\put(-170,73){\footnotesize{$\scriptstyle{r/2}$}}
\put(-180,157){\footnotesize{$\scriptstyle{r/2}$}}
\caption{Des rep\`eres aux futals} 
\end{center}
\end{figure}

Soit $g_r = a_{-r/4} R_{\frac{2i\pi}{3}} a_{r/4}$; on a
$x_{(k+1 ,r)}  =  x_{(k,r)}   g_r$ pour $k = 0, 1,2 $.
On \'ecrit la d\'ecomposition de Cartan $g_r = R_{\theta_1 ( r )+\pi} a_{\ell ( r)} R_{\theta_2 ( r)}$ avec $r \mapsto \ell ( r)$ impaire. Un calcul simple montre alors qu'il existe une constante universelle $C>0$ telle que pour tout $r>0$, les trois nombres
\begin{equation} \label{eq:l}
|\ell ( r) - \frac{r}{2} + \log  \frac{4}{3} | ,  \ | \theta_1 (r )| \mbox{ et } |\theta_2 ( r)| \mbox{ sont inf\'erieurs \`a } C e^{- \frac{r}{4}} .
\end{equation}

\begin{defi} \label{def:bc}
\`A tout triplet $A= (A_k)_{k=0,1,2} \in \Gamma^{3}$
on associe la repr\'esentation 
$\rho_{A} : \pi_1 (\Pi ) \rightarrow \Gamma$ donn\'ee par les formules~: 
$$\rho_{A} (c_n ) = A_{n-1} A_n^{-1}  \quad \mbox{pour } n=1,2,3 .$$
\end{defi}

L'image de $(\gamma_1 , \gamma_2 , \gamma_3) = (\rho_{A} (c_1) , \rho_{A} (c_2) , \rho_{A} (c_3))$ dans $\widehat{\mathcal{P}}$ est uniquement d\'etermin\'ee par la classe de $A$ dans le quotient 
$\Gamma \backslash \Gamma^{3}$. 

\begin{defi}
Soit $dn$ est la mesure de comptage sur $\Gamma^{3}$.
On note $\widehat{\beta}_{\varepsilon , r}$ la mesure sur $\PSL_2 (\C) \times \PSL_2 (\C) \times \Gamma^{3}$ d\'efinie par
$$\prod_{k\in \Z/3 \Z} a_{\varepsilon , r/2} (x_{(-k,-r)} , A_k y_{(k ,r)}) \; d\nu(x) \otimes d\nu (y) \otimes dn(A).$$
\end{defi}

La mesure $\widehat{\beta}_{\varepsilon , r}$ est $\Gamma$-invariante \`a gauche (pour l'action diagonale) et non triviale d'apr\`es la proposition \ref{P:ME}. Elle induit donc une mesure 
non nulle $\beta_{\varepsilon, r}$ sur le quotient $\Gamma \backslash (
\PSL_2 (\C) \times \PSL_2 (\C) \times \Gamma^{3})$. On note 
$$\Phi : \Gamma \backslash (
\PSL_2 (\C) \times \PSL_2 (\C) \times \Gamma^{3}) \rightarrow \widehat{\mathcal{P}}$$
l'application qui \`a la classe de $(x,y,A)$ associe $[\rho_{A}]$.

\begin{prop} \label{P:RP}
Il existe une constante $D>0$ telle que pour tout $\varepsilon >0$, il existe un r\'eel $r_{\varepsilon}$ tel que pour tout $r \geq r_{\varepsilon}$ la mesure $\Phi_* \beta_{\varepsilon , r}$ sur $\widehat{\mathcal{P}}$ est support\'ee dans l'ensemble des futals $(2r - 2\log \frac{4}{3} , D \varepsilon)$-plats. 
\end{prop}
\begin{proof} Soit $(x,y, A)$ dans le support de $\widehat{\beta}_{\varepsilon , r}$.
Il s'agit de calculer les demi-longueurs complexes des \'el\'ements $\gamma_n = \rho_{A} (c_n)$ pour $n=1,2,3$. 
Consid\'erons le cas $n=1$. Puisque $a_{\varepsilon , r/2} (x_{(0, - r)} , A_0 y_{(0,r)})$ est non nul,  il existe $g_0 \in \PSL_2 (\C)$ tel que $d(x_{(0,-r)} , g_0a_{r/2})$ et $d(A_0y_{(0,r)} , g_0)$ soient inf\'erieurs
\`a $\varepsilon/2$. Il existe donc des \'el\'ements $g_{\varepsilon}^{(1)}$ et $g_{\varepsilon}^{(2)}$ dans la boule
de rayon $\varepsilon/2$ autour de l'identit\'e dans $\PSL_2 (\C)$ tels que $x_{(0,-r)} g_{\varepsilon}^{(1)} = g_0 a_{r/2}$ et $g_0 g_{\varepsilon}^{(2)} = A_0 y_{(0,r)}$; on a alors~:
\begin{equation} \label{eq:1}
x_{(0,-r)} g_{\varepsilon}^{(1)} a_{-r/2} g_{\varepsilon}^{(2)} = A_0 y_{(0,r)}.
\end{equation} 
De la m\^eme mani\`ere, puisque 
$a_{\varepsilon , r/2} (x_{(2,-r)}  ,  A_1 y_{(1,r)})$ est non nul, il existe des \'el\'ements $g_{\varepsilon}^{(3)}$ et $g_{\varepsilon}^{(4)}$ dans la boule
de rayon $\varepsilon/2$ autour de l'identit\'e dans $\PSL_2 (\C)$ tels que 
\begin{equation} \label{eq:2}
A_1 y_{(1,r)} g_{\varepsilon}^{(3)} a_{r/2} g_{\varepsilon}^{(4)} = x_{(2,-r)}.
\end{equation}
En utilisant tour \`a tour que $x_{(0,-r)} = x_{(2,-r)} R_{\theta_1 (-r)} R_{\pi} a_{- \ell ( r)} R_{\theta_2 (- r)}$ et que $y_{(1,r)} = y_{(0,r)} R_{\theta_1 ( r)} R_{\pi} a_{\ell ( r)} R_{\theta_2 ( r)}$, on d\'eduit de \eqref{eq:1} et \eqref{eq:2} que
$$A_0 A_1^{-1} g_0 = g_0 ( g_{\varepsilon}^{(2)} R_{\theta_1 ( r)} R_{\pi} a_{\ell ( r)} R_{\theta_2 ( r)} g_{\varepsilon}^{(3 )} a_{r/2} g_{\varepsilon}^{(4)} R_{\theta_1 (-r)} R_{\pi} a_{-\ell ( r)}  R_{\theta_2 (-r)} g_{\varepsilon}^{(1)} a_{-r/2}).$$  
Puisque $R_{\pi} a_t = a_{-t} R_{\pi}$, comme la distance $d$ est $K$-invariante \`a gauche et \`a droite et puisque pour $r$ 
grand les $R_{\theta_i ( r)}$ sont proches de l'identit\'e, on obtient qu'il existe un r\'eel strictement positif $r_{\varepsilon}$ tel que pour tout $r \geq r_{\varepsilon}$ il existe 
des \'el\'ements $h_{\varepsilon}^{(i)}$ pour $i=1, \ldots , 4$, dans la boule
de rayon $\varepsilon$ autour de l'identit\'e dans $\PSL_2 (\C)$, tels que
$$\gamma_1 g_0 = g_0 ( h_{\varepsilon}^{(1)} a_{-\ell ( r)} h_{\varepsilon}^{(2 )} a_{-r/2} h_{\varepsilon}^{(3)} a_{-\ell ( r)} h_{\varepsilon}^{(4)} a_{-r/2}).$$  
Mais un calcul \'el\'ementaire, utilisant \eqref{eq:l}, implique que, quitte \`a augmenter $r_{\varepsilon}$, on peut supposer que pour $r\geq r_{\varepsilon}$, la matrice 
$$h_{\varepsilon}^{(1)} a_{-\ell ( r)} h_{\varepsilon}^{(2 )} a_{-r/2} h_{\varepsilon}^{(3)} a_{-\ell ( r)} h_{\varepsilon}^{(4)} a_{-r/2}$$
est de la forme 
\begin{equation} \label{eq:mat}
\left( \begin{array}{cc}
e^{-2r + 2 \log (4/3)} + O (\varepsilon ) & e^{2r-2\log (4/3)} O(\varepsilon )  \\
O (\varepsilon ) & e^{2r-2\log (4/3)} (1 + O (\varepsilon)) 
\end{array} \right).
\end{equation}
La trace de $\gamma_1$ est donc \'egale \`a $e^{2r-2\log (4/3)} (1+O(\varepsilon))$. En raisonnant de la m\^eme mani\`ere avec $\gamma_2$ et $\gamma_3$ (et en augmentant au besoin $r_{\varepsilon}$), on conclut qu'il existe une constante universelle $D$ telle que pour $r \geq r_{\varepsilon}$ les longueurs complexes de la repr\'esentation $\rho_{A}$ v\'erifient, pour tout $n=1,2,3$,
$$\big| \ell_n (\rho_{A} ) -  2r + 2\log \frac{4}{3} \big| \leq  D \varepsilon .$$
En particulier les \'el\'ements $\rho_{A} (c_n )$ pour $n=1,2,3$ sont loxodromiques. Enfin, puisque la matrice $a_{-2r + 2 \log (4/3)}$ est r\'eelle, un argument de continuit\'e en $\varepsilon$ montre que $\sigma_n (\rho_{A}) = \frac12 \ell_n (\rho_{A} )$ pour $n=1,2,3$. La classe de conjugaison de la repr\'esentation $\rho_{A}$
repr\'esente donc un futal $(2r - 2\log \frac{4}{3} , D\varepsilon)$-plat.
\end{proof}

\begin{rema}\label{rem} 1. Il d\'ecoule des propositions \ref{P:ME} et \ref{P:RP} que pour $R$ suffisamment grand il existe bien des futals $(R, \varepsilon)$-plats. Ce point crucial
est d\^u \`a Lewis Bowen \cite{Bowen} qui \'etait lui aussi motiv\'e par la conjecture des sous-groupes de surfaces.

2. La d\'emonstration montre plus\footnote{C'est d'ailleurs uniquement pour cette raison que l'on a commenc\'e par transporter parall\`element les rep\`eres associ\'es \`a $x$ et $y$ pendant un temps $r/4$, voir la proposition \ref{pPi}.}~: un calcul simple implique en effet que, pour $r$ suffisamment grand, les points fixes $z_{\pm} \in \C \cup \{ \infty \}$ de la matrice \eqref{eq:mat}
v\'erifient $|z_- | = O (\varepsilon )$ et $|z_+ | \geq \frac{1}{O( \varepsilon )}$. Cela montre que -- quitte \`a augmenter les constantes $D$ et $r_{\varepsilon}$ -- on peut supposer que 
pour tout $r \geq r_{\varepsilon}$ l'axe de la matrice \eqref{eq:mat} reste \`a distance au plus $D\varepsilon$ du point base $p_0$ et donc que l'axe de $\gamma_1\in \Gamma$ est \`a distance $\leq D\varepsilon$ de $g_0(p_0)$. Par sym\'etrie on en d\'eduit que l'on peut supposer que l'axe de $\gamma_1$ est \`a distance $\leq D\varepsilon$ des points $x_{(0,-r)} \cdot p_0$, $A_0 y_{(0,r)} \cdot p_0$, $x_{(2 ,-r)} \cdot p_0$ et $A_1 y_{(1,r)} \cdot p_0$ de $\H_3$.    
\end{rema}

\begin{defi} \label{def:mu}
\'Etant donn\'e un r\'eel $\varepsilon >0$ et un r\'eel $R>0$, on pose
$$\mu_{\varepsilon , R} = \Phi_* \beta_{ D^{-1} \varepsilon , \; R/2+ \log \frac{4}{3}}.$$ 
Et on note $R_{\varepsilon} =2 r_{\varepsilon} - 2\log \frac{4}{3}$. 
\end{defi}

Il d\'ecoule de la proposition \ref{P:RP} que pour $R \geq R_{\varepsilon}$, la mesure $\mu_{\varepsilon , R} \in \mathcal{M} (\mathcal{P})$ -- qui est non nulle d'apr\`es la proposition \ref{P:ME} --- est support\'ee
sur les futals $(R, \varepsilon)$-plats. Pour d\'emontrer le th\'eor\`eme \ref{TA}, il reste \`a d\'emontrer la proposition suivante; c'est l'objet des paragraphes qui suivent.

\begin{prop} \label{prop:fin}
Il existe une constante strictement positive $D'$ telle que pour tout $\varepsilon \in \; ]0,1]$, la mesure 
$p_* \mu_{\varepsilon, R}$ sur $\T_M$ est $D'Re^{-qR}$-\'equidistribu\'ee.
\end{prop}

\subsection{L'application $\varpi$}

Soient $\varepsilon$ et $r$ deux r\'eels strictement positifs.  

Soient $\gamma \in \Gamma$ et $Z=Z_{\gamma}$ son centralisateur dans $\PSL_2 (\C)$. 
Un couple $(x,y) \in \PSL_2 (\C)^2$ d\'etermine une g\'eod\'esique orient\'ee dans $\H_3$ allant de $yR_{\frac{4i\pi}{3}} (\infty) \in \bP_1 (\C)$ \`a $xR_{\frac{2i\pi}{3}} (0) \in \bP_1 (\C)$.\footnote{Noter que $yR_{\frac{4i\pi}{3}} (\infty) = \lim_{r \rightarrow +\infty} y_{(2,r)}$ et $xR_{\frac{2i\pi}{3}} (0) 
= \lim_{r \rightarrow +\infty } x_{(1,-r)}$.}
Si l'axe de translation de $\gamma$ n'intersecte pas cette g\'eod\'esique la perpendiculaire commune d\'etermine un \'el\'ement du fibr\'e unitaire normal \`a l'axe de $\gamma$ et donc, comme au \S \ref{par:FN}, un rep\`ere de $\H_3$ au-dessus de l'axe de $\gamma$. On \'ecrit $\varpi_{\gamma}(x,y) \cdot (p_0, \vec{u}_0 , \vec{n}_0 , \vec{u}_0 \wedge \vec{n}_0)$, avec $\varpi_{\gamma} (x,y) \in Z$, ce rep\`ere. En tant que groupe, $Z$ op\`ere sur lui-m\^eme et diagonalement (\`a gauche) sur $\PSL_2 (\C)^2$. Si $z \in Z$, on a 
$\varpi_{\gamma} (zx,xy) = z \varpi_{\gamma} (x,y)$. L'application $\varpi_{\gamma}$ passe donc au quotient par le groupe $\langle \gamma \rangle$ (dont l'action commute \`a celle de $Z$) en 
une application $Z$-\'equivariante, d\'efinie presque partout, 
$$\varpi_{\gamma} : \langle \gamma \rangle \backslash (\PSL_2 (\C) \times \PSL_2 (\C)) \rightarrow \widehat{\T}_{\gamma} = \langle \gamma \rangle \backslash Z.$$
Noter que si $h \in \PSL_2 (\C)$, 
la conjugaison par $h$ identifie $\widehat{\T}_{\gamma}$ et $\widehat{\T}_{h\gamma h^{-1}}$ et on a~:
\begin{equation} \label{eq:Gequiv}
\varpi_{h\gamma h^{-1}} (hx,hy) = h \varpi_{\gamma} (x,y) h^{-1}.
\end{equation}

\begin{defi}
On note $\alpha_{\gamma}=\alpha_{\varepsilon , r, \gamma}$ la mesure sur $\langle \gamma \rangle \backslash (\PSL_2 (\C) \times \PSL_2 (\C))$ induite par la mesure $\gamma$-invariante (\`a gauche)
$$a_{\varepsilon , r/2} (x_{(0,-r)} , y_{(0,r)}) a_{\varepsilon , r/2} ( \gamma x_{(2,-r)} , y_{(1,r)}) d\nu(x) \otimes d\nu(y).$$
\end{defi}

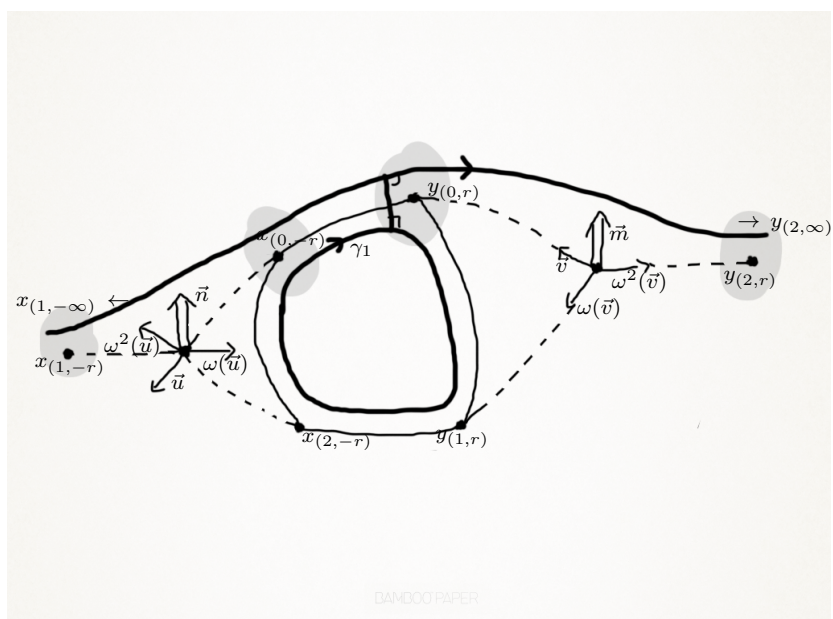
\begin{figure}[ht]      
\begin{center}
\includegraphics[scale=.3]{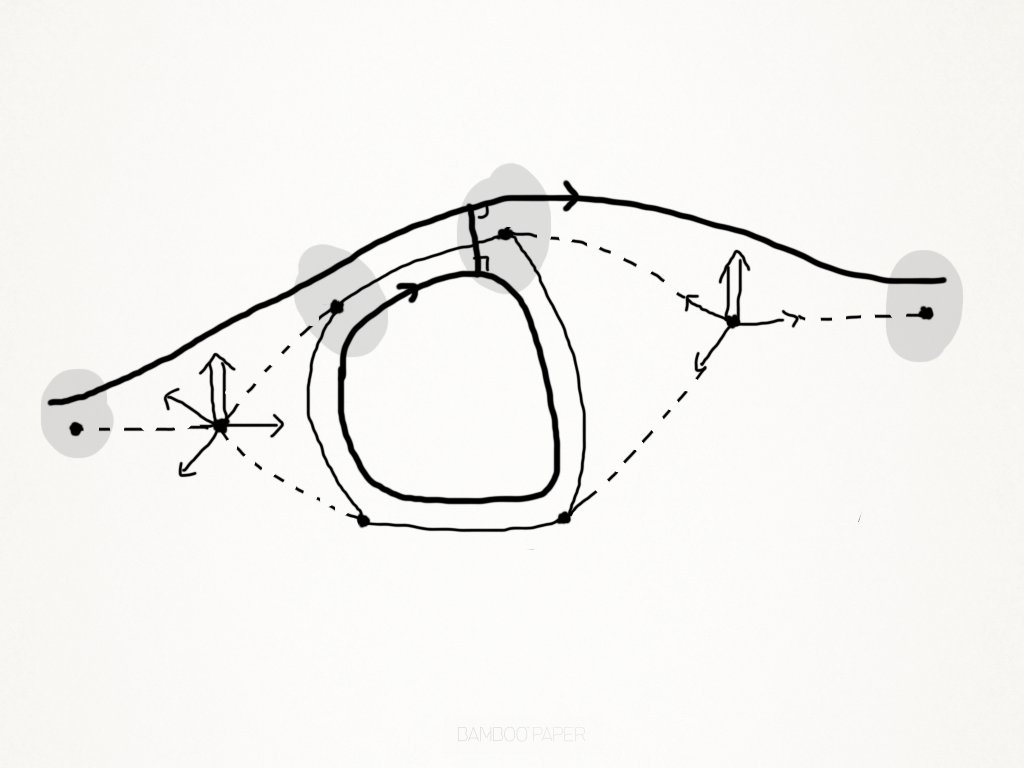}
\put(-83,145){$\scriptstyle{\vec{m}}$}
\put(-82,127){$\scriptstyle{\omega^2 (\vec{v} )}$}
\put(-103,133){$\scriptstyle{\vec{v}}$}
\put(-95,117){$\scriptstyle{\omega (\vec{v} )}$}
\put(-35,150){$\scriptstyle{\rightarrow \  y_{(2,\infty)}}$}
\put(-40,130){$\scriptstyle{y_{(2,r)}}$}
\put(-180,140){$\scriptstyle{\gamma_{1}}$}
\put(-235,96){$\scriptstyle{\omega (\vec{u} )}$}
\put(-247,89){$\scriptstyle{\vec{u}}$}
\put(-272,103){$\scriptstyle{\omega^2 (\vec{u} )}$}
\put(-238,122){$\scriptstyle{\vec{n}}$}
\put(-298,97){$\scriptstyle{x_{(1,-r)}}$}
\put(-198,69){$\scriptstyle{x_{(2,-r)}}$}
\put(-148,70){$\scriptstyle{y_{(1,r)}}$}
\put(-215,145){$\scriptstyle{x_{(0,-r)}}$}
\put(-151,163){$\scriptstyle{y_{(0,r)}}$}
\put(-305,120){$\scriptstyle{x_{(1,-\infty)} \ \leftarrow}$}
\caption{Pied abstrait} \label{Pied}
\end{center}
\end{figure}

Noter que la mesure $\alpha_{\gamma}$ est  $Z$-invariante de sorte que 
\begin{equation} \label{eq:inv}
(\varpi_{\gamma})_* \alpha_{\gamma} = \Lambda_{(\varpi_{\gamma})_* \alpha_{\gamma}}= \left[ \frac{(\varpi_{\gamma})_* \alpha_{\gamma} (\widehat{\T}_{\gamma})}{\Lambda (\widehat{\T}_{\gamma})} \right] \Lambda,
\end{equation}
o\`u $\Lambda$ est la mesure de Lebesgue sur le tore $\widehat{\T}_{\gamma}$. 

\begin{defi}
Soit 
$$\varpi : \Gamma \backslash (\PSL_2 (\C) \times \PSL_2 (\C) \times \Gamma^{3}) \rightarrow \T_{M}$$
l'application qui \`a la classe de $(x,y, A)$ associe l'image de $\varpi_{\gamma_1} (x , A_0 y)$
dans le tore abstrait $\T_g$ o\`u 
$g$ est la g\'eod\'esique ferm\'ee de $M$ associ\'ee \`a $\gamma_1 = A_0 A_1^{-1}$; voir figure~\ref{Pied}.
\end{defi}

Noter que l'application $\varpi$ est bien d\'efinie~: changer $(x,y,A)$ en $\gamma \cdot (x,y, A)$ revient
\`a changer $\varpi_{\gamma_1} (x , A_0 y)$ en $\varpi_{\gamma \gamma_1 \gamma^{-1}} (\gamma x , \gamma A_0 y )$ qui -- en vertu de \eqref{eq:Gequiv} -- a la m\^eme image dans le tore euclidien $\T_g$.

L'application qui \`a un futal marqu\'e associe son pied induit \'egalement une application 
$$p \circ \Phi : \Gamma \backslash (\PSL_2 (\C) \times \PSL_2 (\C) \times \Gamma^{3}) \rightarrow \T_{M}.$$ 

\begin{prop} \label{pPi}
Il existe des constantes $c,K>0$ telles que pour tout $\varepsilon \in \; ]0,1]$, il existe un r\'eel $r_{\varepsilon}'$ tel que pour tout $r \geq r_{\varepsilon}'$ et pour tout \'el\'ement $[x,y,A]$ dans le support de $\beta_{\varepsilon , r}$, on a~:
$$d( p\circ \Phi (x,y,A) , \varpi (x,y,A) ) \leq K e^{-cr}.$$
\end{prop}
\begin{proof} La proposition r\'esulte du fait g\'en\'eral que dans un espace hyperbolique, pour tout $M>0$, il existe des constantes $c,K>0$ telles que deux g\'eod\'esiques qui restent \`a distance $\leq M$ pendant un temps $r$ sont $Ke^{-cr}$-proches au temps $r/2$; voir figure \ref{Pied}. On renvoie \`a \cite[Lem. 4.2]{KM1} pour plus de d\'etails et des constantes  explicites.   
\end{proof} 

\begin{coro} \label{Cor}
Les mesures $(p\circ \Phi)_* \beta_{\varepsilon , r}$ et  $\varpi_* \beta_{\varepsilon , r}$ sont $K e^{-cr}$-proches. 
\end{coro}

\subsection{D\'emonstration de la proposition \ref{prop:fin}}

Soit $\gamma \in \Gamma$. On note $C_{\gamma}$ le sous-ensemble de $\PSL_2 (\C) \times \PSL_2 (\C) \times \Gamma^{3}$ constitu\'e des triplets $(x,y,A)$ tels que $A_0 A_1^{-1} = \gamma$. Soit $\chi : C_{\gamma} \rightarrow \PSL_2 (\C) \times \PSL_2 (\C)$ l'application $(x,y,A) \mapsto (x,A_0^{-1} y)$. Le lemme suivant d\'ecoule de la d\'efinition de $ A_{\varepsilon , r/2}$.

\begin{lemm} \label{L2}
On a~: 
\begin{multline*}
\chi_* (\widehat{\beta}_{\varepsilon , r} |_{C_{\gamma}}) 
\\ = A_{\varepsilon , r/2} (x_{(1,-r)} , y_{(2,r)}) a_{\varepsilon , r/2} (x_{(0,-r)} , y_{(0,r)}) a_{\varepsilon , r/2} ( \gamma x_{(2,-r)} , y_{(1,r)}) d\nu(x) \otimes d\nu(y).
\end{multline*}
\end{lemm}

La mesure $\chi_* (\widehat{\beta}_{\varepsilon , r} |_{C_{\gamma}})$ 
sur $\PSL_2 (\C )^2$ est $\gamma$-invariante; on note $\beta_{\gamma}$ la mesure induite sur le quotient $\langle \gamma \rangle \backslash \PSL_2 (\C)^2$. 
On peut maintenant pousser -- \`a l'aide de l'application $\varpi_{\gamma}$ -- les mesures $\alpha_{\gamma}$  et 
$\beta_{\gamma}$ en des mesures sur 
le tore $\widehat{\T}_{\gamma}$. D'apr\`es la d\'efinition de $\alpha_{\gamma}$, le lemme \ref{L2} et la proposition \ref{P:ME}, on a~:
$$(1 - Ce^{-qr}) (\varpi_{\gamma} )_* \alpha_{\gamma} \leq (\varpi_{\gamma} )_* \beta_{\gamma} \leq  (1 + Ce^{-qr}) (\varpi_{\gamma} )_* \alpha_{\gamma} .$$
Mais, pouss\'ee sur le tore abstrait $\T_g$, o\`u $g$ est la g\'eod\'esique ferm\'ee de $M$ associ\'ee \`a $\gamma$, les mesures $(\varpi_{\gamma} )_* \beta_{\gamma}$ 
et $\left[ \varpi_* \beta_{\varepsilon , r} \right] |_{\T_g}$ co\"{\i}ncident.
Il d\'ecoule alors de \eqref{eq:inv} que l'on a~:
\begin{equation} \label{eq:leb}
(1 - Ce^{-qr}) \Lambda_{(\varpi_{\gamma})_* \alpha_{\gamma}} \leq \left[ \varpi_* \beta_{\varepsilon ,  r} \right] |_{\T_g} \leq  (1 + Ce^{-qr})  \Lambda_{(\varpi_{\gamma})_* \alpha_{\gamma}}.
\end{equation}

Le lemme suivant est 
\'el\'ementaire; voir \cite[Lem. 3.1]{KM1}.

\begin{lemm} \label{L3}
Soient $a,b \in \C$ tels que $\T = \C / (a \Z + ib \Z)$ soit un tore. Soit $f$ une fonction continue et positive sur $\T$ et $\delta \in \; ]0, \frac{1}{3} [$ tels que
$$(1- \delta) \Lambda_{f \Lambda} \leq f \Lambda \leq (1+\delta ) \Lambda_{f \Lambda}.$$
Alors, la mesure $f \Lambda$ est $4 \delta (|a| + |b|)$-\'equidistribu\'ee.
\end{lemm} 

D'apr\`es la proposition \ref{P:RP}, un tore $\T_g = (\R_+^* + i \R)/ (\sigma \Z + 2i\pi \Z)$ 
qui intersecte non trivialement le support de $ \varpi_* \beta_{\varepsilon ,  r}$ v\'erifie $|\sigma | \leq 2r$. Il d\'ecoule donc du lemme \ref{L3} et de l'\'equation \eqref{eq:leb} que la mesure  $\varpi_* \beta_{\varepsilon ,  r}$  est $8rCe^{-qr}$-\'equidistribu\'ee. Le corollaire \ref{Cor} implique alors que pour $r$ sufisamment grand (et en supposant $q\leq c$) les mesures $(p \circ \Phi)_* \beta_{\varepsilon ,  r}$  et $\Lambda_{ (p \circ \Phi)_* \beta_{\varepsilon ,  r}}$ sont $9rCe^{-qr}$-proches. Puisque -- compte tenu de la d\'efinition \ref{def:mu} -- 
on a $p_*\mu_{\varepsilon, R}= (p \circ \Phi)_*  \beta_{D^{-1} \varepsilon , R/2 + \log \frac{4}{3}}$, cela suffit \`a d\'emontrer la  proposition \ref{prop:fin}. 

\section{D\'emonstration du th\'eor\`eme \ref{TB}}

Si $\varepsilon$ est suffisamment petit, une repr\'esentation $(R, \varepsilon)$-plate est la monodromie
d'une structure de $\PSL_2 (\C)$-surface. Si $\varepsilon =0$ cette structure est en fait une structure hyperbolique r\'eelle; nous dirons alors que c'est une {\it surface de type $R$}.    

\subsection{G\'eom\'etrie des surfaces de type $R$}
Une surface $S$ de type $R$ est munie d'une d\'ecomposition en pantalons. On appelle {\it revers} de cette d\'ecomposition les courbes simples ferm\'ees g\'eod\'esiques qui 
d\'ecoupent $S$ en pantalons. Chaque revers est de longueur $R$ et chaque param\`etre de d\'ecalage est \'egal \`a $1$. 

Dans un pantalon hyperbolique dont toutes les composantes de bord sont de longueur $R$, la distance entre deux composantes de bord est \'egale \`a $2e^{-R/4} + O (e^{-3R/4} )$, lorsque $R$ tend vers l'infini. 
Il d\'ecoule donc de la {\it formule du pentagone} \cite[Thm. 2.3.4]{Buser}
que, pour $R$ suffisamment grand, une courbe de longueur $\leq R/5$ joignant deux composantes de bord est homotope, relativement au bord, \`a une g\'eod\'esique minimisante entre ces deux composantes de bord. 

\begin{figure}[ht]      
\begin{center}
\includegraphics[scale=.2]{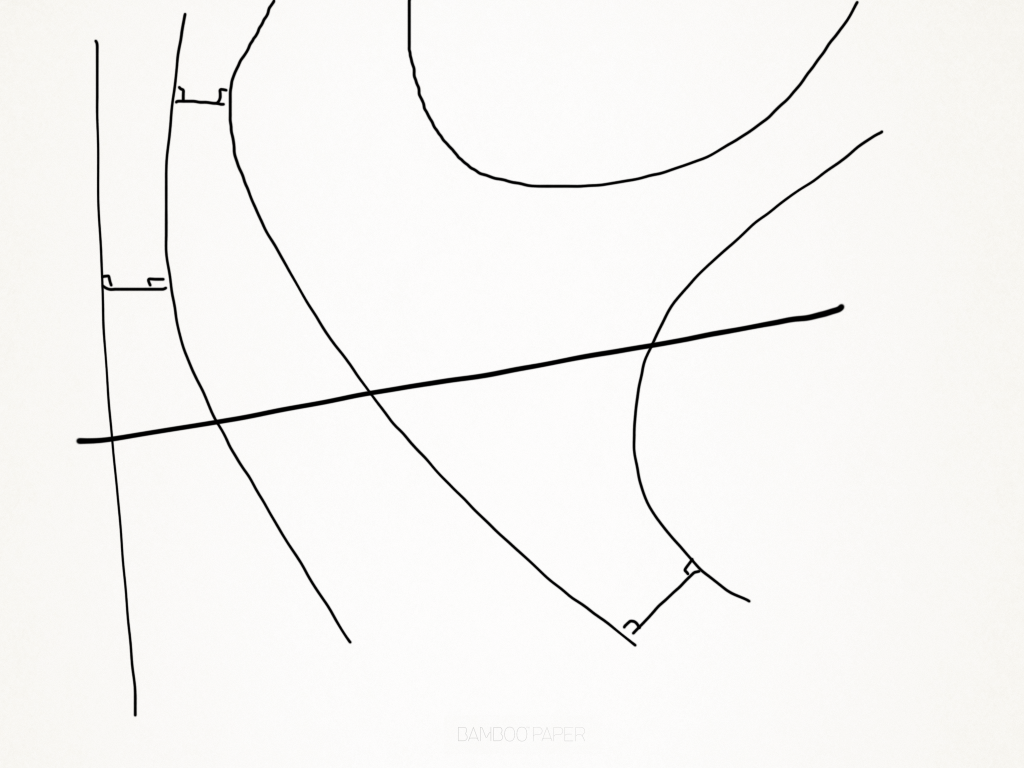}
\put(-180,5){$\scriptstyle{C_0}$}
\put(-135,20){$\scriptstyle{C_1}$}
\put(-78,18){$\scriptstyle{C_2}$}
\put(-53,28){$\scriptstyle{C_3}$}
\put(-100,76){$\scriptstyle{g}$}
\put(-197,95){$\scriptstyle{w_0^-}$}
\put(-170,95){$\scriptstyle{w_0^+}$}
\put(-182,133){$\scriptstyle{w_1^-}$}
\put(-158,133){$\scriptstyle{w_1^+}$}
\put(-92,23){$\scriptstyle{w_2^-}$}
\put(-65,40){$\scriptstyle{w_2^+}$}
\caption{Le segment $g$ et les $C_i$} \label{Fig:5}
\end{center}
\end{figure}
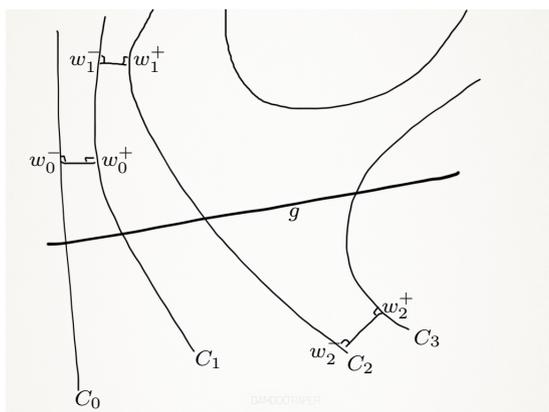

Soient $g$ un segment g\'eod\'esique de $\H_2$ et $C_0, \ldots , C_k$ l'ensemble ordonn\'e des g\'eod\'esiques (bi-infinies) de $\H_2$ qui se projettent sur les revers de $S$
et qui intersectent $g$, voir la figure \ref{Fig:5}. Alors deux g\'eod\'esiques adjacentes $C_{i}$ et $C_{i+1}$
sont soit {\it proches}, \`a distance $\approx 2e^{-R/4}$, soit {\it \'eloign\'ees}, \`a distance $> R/5$. Pour $i=0, \ldots , k$ on note $z_i$ le point d'intersection de $g$ et $C_i$ et $w_i^+$ le
pied dans $C_{i+1}$ de la perpendiculaire commune \`a $C_i$ et $C_{i+1}$. 

Dans le plan hyperbolique, les g\'eod\'esiques s'\'eloignent \`a vitesse lin\'eaire (voir \cite[Thm. 2.3.1]{Buser})~: pour $i=0, \ldots , k-1$, on a
$$d(C_{i} , C_{i+1} ) e^{d(w_i^+ , z_{i+1})} \leq e^{d(z_{i+1} , C_i)}.$$
Le fait que le param\`etre de d\'ecalage le long de chaque revers soit \'egal \`a $1$ implique donc 
que pour $R$ suffisamment grand si, pour $i=1, \ldots , k$, les g\'eod\'esiques $C_{i-1}$ et $C_i$ sont proches et $d(z_{i-1} , z_i) \leq 1$ alors $k\leq R$. En distinguant les sous-segments de $g$ qui correspondent \`a une suite $(z_i )_{i=0 , \ldots k}$ telle que $d(z_{i-1} , z_i ) \leq 1$ et les \og sauts \fg \ correspondants aux sous-segments compl\'ementaires, on obtient la proposition suivante, voir \cite[Lem. 2.3]{KM1}.

\begin{prop} \label{P:SR}
Il existe un r\'eel $R_0$ et une constante $C$ tels que pour tout $R \geq R_0$, tout segment g\'eod\'esique de longueur $\ell$ sur une surface $S$ de type $R$ coupe au plus $CR \ell$ fois les revers.
\end{prop}

\subsection{D\'emonstration du th\'eor\`eme \ref{TB}}

Une repr\'esentation $(R , \varepsilon)$-plate $\rho$ est la monodromie d'une structure de $\PSL_2 (\C)$-surface sur une surface $S$ munie d'une d\'ecomposition en pantalons.
On muni $S$ de la structure hyperbolique de type $R$ associ\'ee \`a cette d\'ecomposition. La surface $S$ est alors uniformis\'ee par le plan hyperbolique $\H_2$ et on note 
$F : \H_2 \rightarrow \H_3$ l'application d\'eveloppante de la structure de $\PSL_2 (\C )$-surface sur $S$ associ\'ee \`a $\rho$. Puisque $F$ est $\rho$-\'equivariante, pour 
d\'emontrer le th\'eor\`eme \ref{TB}, il suffit de montrer qu'il existe des r\'eels strictement positifs
$\varepsilon_0$ et $R_0$ tels que si $R \geq R_0$ et $\varepsilon \in ]0, \varepsilon_0]$ alors $F$ envoie toute g\'eod\'esique de $\H_2$ sur une courbe bi-infinie dans $\H_3$.

Mais l'application $F$ envoie les g\'eod\'esiques $C_i$ de $\H_2$ au-dessus des revers de la surface $S$ (de type $R$) sur des g\'eod\'esiques de $\H_3$. Il suffit donc de montrer que si 
$g$ est une g\'eod\'esique de $\H_2$ transversale aux $C_i$ alors $F(g)$ est une courbe bi-infinie de $\H_3$. Comme dans la d\'emonstration de la proposition \ref{P:SR}, on d\'ecoupe $g$ en deux types de sous-segments~: ceux qui rencontrent une suite de $C_i$ proches et les \og sauts \fg.  L'image d'un saut est un long segment quasi-g\'eod\'esique, dont la constante de quasi-g\'eod\'esie ne d\'epend que de $\varepsilon$ et pas de $R$. La difficult\'e consiste donc \`a contr\^oler les segments g\'eod\'esiques de $C$ qui correspondent \`a une suite de $C_i$ proches. Mais d'apr\`es la proposition \ref{P:SR}, une telle suite contient au plus $R$ \'el\'ements alors que le long des revers l'angle de d\'ecalage est $\leq \varepsilon/R$. L'image d'un tel segment est donc l\`a encore un segment quasi-g\'eod\'esique, dont la constante de quasi-g\'eod\'esie ne d\'epend que de $\varepsilon$ et pas de $R$. Gr\^ace \`a l'hyperbolicit\'e de $\H_3$, on conclut que la courbe $F(g)$ est quasi-g\'eod\'esique; elle
est en particulier bi-infinie.

\begin{rema} Les d\'etails de cette \og d\'emonstration \fg \ sont un peu laborieux \`a \'ecrire. Noter n\'eanmoins que le m\'elange exponentiel permet en fait de construire des surfaces $(R, e^{-dR})$-plates pour une constante $d=d(M)>0$. Il est alors plus facile de faire fonctionner l'approche d\'ecrite ci-dessus; Dragomir Sarik \cite{Sarik} r\'ealise d'ailleurs toute surface $(R, \varepsilon/R)$-plate comme surface pliss\'ee.
\end{rema}

Dans \cite{KM1} Kahn et Markovic d\'emontrent le th\'eor\`eme \ref{TB} \`a l'aide d'une formule, due \`a Caroline Series \cite{Series}, de variation de la distance 
$d(C_0 , C_k )$ le long d'une d\'eformation quasi-fuchsienne $t \mapsto \rho_t$. Ils montrent que le long d'une d\'eformation $(R, \varepsilon)$-plate la d\'eriv\'ee $d(C_0, C_k) '$ est 
un $O(\varepsilon)$. Cela leur permet de montrer que, quitte \`a diminuer $\varepsilon_0$, on peut supposer que l'image du bord $\partial \H_2$, par l'application induite $\partial F: \partial \H_2 \rightarrow \partial \H_3$, est arbitrairement proche d'un vrai cercle dans $\partial \H_3$. Enfin puisque l'on peut d\'emarrer la construction \`a partir de deux rep\`eres appartenant \`a un plan arbitraire de $\H_3$, la d\'emonstration du th\'eor\`eme \ref{TA} implique en fait le th\'eor\`eme suivant.

\begin{theo} \label{Tprinc}
Soit $M=\Gamma \backslash \H_3$ une vari\'et\'e hyperbolique uniformis\'ee de dimension $3$. Pour tout cercle $C$ dans $\partial \H_3$, il existe une suite de $\PSL_2 (\C)$-surfaces $(F_n : \widetilde{S}_n \rightarrow \H_3 )_{n\geq 0}$ dans $M$ telle que $(\partial F_n (\widetilde{S}_n ) )_{n \in \N}$ converge vers $C$, pour la distance de Hausdorff dans $\partial \H_3$.
\end{theo}

\section{Applications et g\'en\'eralisations}

\subsection{Du th\'eor\`eme \ref{T1} \`a la conjecture VH}

Soit $M$ une vari\'et\'e hyperbolique de dimension $3$. 
D'apr\`es le th\'eor\`eme \ref{T1}, il existe une surface $S$ de genre $g\geq 2$ et une immersion $\pi_1$-injective $f : S \rightarrow M$. 

Peter Scott \cite{Scott} donne un crit\`ere alg\'ebrique d'existence d'un rev\^etement fini de $M$ auquel $f$ se rel\`eve en un plongement. 
 
\begin{defi}
Soit $G$ un groupe de type fini. Un sous-ensemble $S$ de $G$ est {\rm s\'eparable} si $S$ est 
ferm\'e dans la topologie profinie de $G$, c'est-\`a-dire si
$S$ est une intersection de classes \`a gauche de sous-groupes d'indice fini de $G$.
\end{defi}
 
\begin{prop}
Supposons $f_* (\pi_1 S)$ s\'eparable dans $\pi_1 M$. Alors $f$ se rel\`eve \`a un rev\^ement fini 
de $M$ en un plongement, ou en une application isotope \`a un plongement.  
\end{prop}

Si de plus $f_* (\pi_1 S)$ est quasi-convexe dans $\pi_1M$, alors il existe un rev\^etement fini $\widehat{M}$ de $M$ qui contient deux \'el\'evations 
plong\'ees et disjointes de $f(S)$ dont la r\'eunion est non s\'eparante, voir par exemple \cite[Th\'eor\`eme 2]{EnsMath}. En particulier le groupe $\pi_1 (\widehat{M})$ se surjecte sur un groupe libre de rang $2$.
En plus des conjectures mentionn\'ees en introduction, il convient donc d'ajouter l'\'enonc\'e suivant\footnote{\'Enonc\'e \'egalement conjectur\'e dans \cite{Thurston}.} qui fait le lien entre la conjecture des sous-groupes de surfaces et la conjecture VH.

\begin{enumerate}
\item[7.] Tout sous-groupe de type fini dans $\pi_1 M$ est s\'eparable.
\end{enumerate}

Scott \cite{Scott} puis Ian Agol, Darren Long et Alan Reid \cite{ALR} montrent qu'un sous-groupe quasi-convexe d'un groupe de reflexions dans un poly\`edre hyperbolique \`a angles droits est s\'eparable. 
Ce r\'esultat est \'etendu au cas des sous-groupes quasi-convexes d'un groupe de Coxeter -- ou d'un groupe d'Artin -- \`a angles droits abstrait par Haglund \cite{Haglund}. Depuis une quinzaine d'ann\'ees, Dani Wise met en {\oe}uvre un vaste programme visant \`a d\'emontrer
la conjecture (7) et plus g\'en\'eralement \`a d\'emontrer la s\'eparabilit\'e des sous-groupes quasi-convexes dans certains groupes $G$ hyperboliques (au sens de Gromov), \`a l'aide de
la m\'ethode de Scott. Un groupe de Coxeter -- ou d'Artin -- \`a angles droits abstrait poss\`ede une r\'ealisation g\'eom\'etrique qui est un complexe cubique \`a courbure n\'egative. La premi\`ere \'etape du programme de Wise passe alors par la \og cubulation \fg \ du groupe $G$. On renvoie \`a \cite{BH} pour une introduction aux complexes cubiques et \`a la notion de courbure 
n\'egative dans ce contexte.

\subsection{Cubulation des vari\'et\'es hyperboliques; groupes sp\'eciaux} \label{par:6.1}

Dans \cite{BW} avec Wise, et ind\'ependamment Guillaume Dufour dans sa th\`ese \cite{Dufour}, on d\'eduit du th\'eor\`eme \ref{Tprinc} et des travaux de Michah Sageev \cite{Sageev} le r\'esultat suivant.

\begin{theo} \label{T:BW}
Soit $M$ une vari\'et\'e hyperbolique de dimension $3$. Alors $\pi_1 M$ op\`ere librement, proprement
avec quotient compact sur un complexe cubique CAT$(0)$ localement fini.
\end{theo}

\begin{defi}
Un complexe cubique $C$ connexe de courbure n\'egative est {\rm sp\'ecial} s'il existe une isom\'etrie locale
de $C$ dans le complexe cubique associ\'e \`a un groupe d'Artin \`a angles droits. On dit alors que
$\pi_1 C$ op\`ere {\rm sp\'ecialement} sur tout rev\^etement universel de $C$.
\end{defi}

La notion de complexe cubique sp\'ecial est due \`a Fr\'ed\'eric Haglund et Dani Wise, voir \cite{HaglundWiseSpecial}
o\`u elle est d\'efinie en termes de configurations interdites pour les hyperplans immerg\'es.
Haglund et Wise d\'emontrent notamment que si $C$ est un complexe cubique sp\'ecial compact et si $\pi_1C$ est un groupe hyperbolique au sens de Gromov, alors tout sous-groupe quasi-convexe de $\pi_1 C$ est s\'eparable.

Ian Agol \cite{Agol} montre par ailleurs que si $M$ est une vari\'et\'e de dimension $3$ dont le groupe
fondamental est isomorphe au groupe fondamental d'un complexe cubique sp\'ecial alors $M$ poss\`ede un rev\^etement fini qui fibre sur le cercle. 

Dans \cite{HaglundWiseSpecial} puis \cite{WiseIsraelHierarchy} Haglund et Wise d\'emontrent des crit\`eres puissants pour qu'un complexe cubique CAT$(0)$ poss\`ede un rev\^etement fini sp\'ecial. Tout r\'ecemment, en se reposant sur ces travaux, Agol \cite{AgolSpecial} a annonc\'e la d\'emonstration du th\'eor\`eme suivant.

\begin{theo}
Soit $G$ un groupe hyperbolique au sens de Gromov qui op\`ere librement, proprement
avec quotient compact sur un complexe cubique CAT$(0)$ localement fini $C$. Alors $G$ 
contient un sous-groupe d'indice fini qui op\`ere sp\'ecialement sur $C$.
\end{theo}

\begin{coro}
Soit $M$ une vari\'et\'e hyperbolique de dimension $3$. Alors $M$ v\'erifie les conjectures (2) -- (7).
\end{coro} 

\subsection{D\'enombrement des surfaces incompressibles}

Soit $M$ une vari\'et\'e hyperbolique de dimension $3$. 
De m\^eme que l'argument ergodique de Margulis lui a permis de retrouver (et d'\'etendre en courbure variable) le \og th\'eor\`eme des g\'eod\'esiques premi\`eres \fg \ sur le comptage des g\'eod\'esiques dans une surface hyperbolique, Kahn et Markovic \cite{KM2} d\'eveloppent leurs m\'ethodes ainsi que des r\'esultats ant\'erieurs de Joseph Masters \cite{Masters} pour compter le nombre $c(M,g)$ de classes de conjugaison de sous-groupes de $\pi_1 M$ isomorphes au groupe fondamental d'une surface de genre $g$.  

\begin{theo} On a~:
$$\lim_{g \rightarrow +\infty} \frac{\log c(M,g)}{2g \log g} = 1 .$$
\end{theo}

\subsection{G\'en\'eralisation \`a tous les espaces sym\'etriques et conjecture d'Ehrenpreis}

Soient $G$ un groupe de Lie r\'eel semi-simple connexe de centre trivial et sans facteur compact, $K$ un sous-groupe compact maximal de $G$ et $X=G/K$ l'espace sym\'etrique associ\'e. 

\begin{conj} \label{CLM}
Soit $\Gamma$ un r\'eseau uniforme dans $G$. Alors $\Gamma$ contient un sous-groupe isomorphe
au groupe fondamental d'une surface
de genre au moins $2$.
\end{conj}

Il est tentant de chercher \`a \'etendre la m\'ethode de Kahn et Markovic \`a ce cadre et plus g\'en\'eralement, partant d'un plongement $\PSL_2 (\R) \subset G$ correspondant \`a un plongement totalement g\'eod\'esique de $\H_2$ dans $X$, de chercher \`a construire des surfaces immerg\'ees $S \rightarrow \Gamma \backslash X$ dont un relev\'e $\widetilde{S} \rightarrow X$ soit arbitrairement proche de l'image de $\H_2$. Si $X$ est de rang r\'eel \'egal \`a $1$, la d\'emonstration du th\'eor\`eme \ref{TB} s'\'etend naturellement. Le rang sup\'erieur semble plus difficile.
En ce qui concerne le th\'eor\`eme \ref{TA}, notons qu'un aspect crucial de sa d\'emonstration est le fait que le centralisateur d'une isom\'etrie hyperbolique poss\`ede un sous-groupe compact {\it connexe} qui contient la rotation d'angle $\pi$ autour de l'axe de translation; cela permet en effet de s'assurer 
que les futals $(R, \varepsilon)$-plats sont \'equidistribu\'es autour d'une g\'eod\'esique donn\'ee. 

Dans \cite{KM3} Kahn et Markovic consid\`erent cependant le cas o\`u $\PSL_2 (\R)$ est plong\'e diagonalement dans $G=\PSL_2 (\R) \times \PSL_2 (\R)$ et $\Gamma$ est un produit de deux groupes fuchsiens. Ils d\'emontrent le th\'eor\`eme suivant initialement conjectur\'e par Leon Ehrenpreis \cite{Ehrenpreis}.

\begin{theo}
Soient $S$ et $T$ deux surfaces de Riemann compactes connexes de genre au moins $2$. Pour tout $K>1$, il existe des rev\^etements
finis $\widehat{S}$ de $S$ et $\widehat{T}$ de $T$, et un hom\'eomorphisme $f: \widehat{S} \rightarrow \widehat{T}$ tel que $f$ soit $K$-quasi-conforme, c'est-\`a-dire que pour tout $x \in \widehat{S}$, on a~:
$$H_f (x) = \limsup_{r \rightarrow 0} \frac{\sup \{ d( f(z) , f(x)) \; : \; d(x,z) = r \}}{\inf \{ d( f(z) , f(x)) \; : \; d(x,z) = r \}} \leq K.$$ 
\end{theo}

Ils montrent plus pr\'ecisemment qu'\'etant donn\'e une surface hyperbolique $S$ et un r\'eel strictement positif $\varepsilon$, pour tout r\'eel $R$ suffisamment grand, il existe une surface hyperbolique (r\'eelle) $(R,\varepsilon)$-plate immerg\'ee dans $S$. La d\'emonstration est dans un premier temps identique \`a celle du th\'eor\`eme \ref{TA}
mais cette fois le fibr\'e normal \`a une g\'eod\'esique ferm\'ee dans $S$ n'est pas connexe. Il se peut donc
qu'il y ait plus de futals immerg\'es d'un c\^ot\'e de la g\'eod\'esique que de l'autre. Il s'agit alors de
corriger ce d\'efaut par un poids dont l'existence est le c{\oe}ur des arguments de \cite{KM3}, voir \'egalement \cite{KM4} pour un premier r\'esultat sur cette question.

\section*{Appendice~: m\'elange exponentiel du flot des rep\`eres}

Dans cet appendice, $G$ est un groupe de Lie simple r\'eel, connexe, non compacte et de centre trivial dont on fixe une d\'ecomposition d'Iwasawa $G=NAK$ et une d\'ecomposition de Cartan compatible $G=KA^+ K$. On note $\mathfrak{n}$
l'alg\`ebre de Lie de $N$. 

\begin{defi}
1. Une {\rm repr\'esentation unitaire} $\pi$ de $G$ dans un espace de Hilbert (s\'eparable) $\mathcal{H}_{\pi}$ est un morphisme $G \rightarrow U(\mathcal{H}_{\pi})$ tel que pour tout 
$v \in \mathcal{H}_{\pi}$ l'application $G \rightarrow \mathcal{H}_{\pi}$; $g \mapsto \pi (g) v$ soit continue. Si cette application est lisse on dit que $v$ est un {\rm vecteur $C^{\infty}$} de $\pi$; on note $\mathcal{H}_{\pi}^{\infty}$ l'ensemble des vecteurs $C^{\infty}$ de $\pi$. 

2. \'Etant donn\'e deux vecteurs $v, w \in \mathcal{H}_{\pi}$, on appelle {\rm coefficient} de $\pi$ la fonction continue $c_{v,w} :G \rightarrow \C$ d\'efinie par $c_{v,w} : g \mapsto \langle \pi (g) v , w \rangle$. On dit que le coefficient $c_{v,w}$ est {\rm $K$-fini} si les espaces vectoriels engendr\'es par respectivement $\pi (K) \cdot v$ et $\pi (K) \cdot w$ sont de dimension finie.

3. On note $p(\pi)$ la borne inf\'erieure de tous les $p \geq 2$ tels que les coefficients $K$-finis de $\pi$ soient dans $L^p (G)$.

4. Une repr\'esentation unitaire $\sigma$ est {\rm faiblement contenue} dans $\pi$ si tout coefficient de $\sigma$
est limite uniforme sur tout compact de coefficients de $\pi$. On dit que $\pi$ a un {\rm trou spectral} si elle ne contient pas faiblement la repr\'esentation triviale de $G$.
\end{defi}

Le th\'eor\`eme de Howe-Moore affirme que si $\pi$ est une repr\'esentation unitaire de $G$ sans vecteurs invariants alors les coefficients  $g \mapsto c_{v,w} (g)$ de $\pi$ tendent vers $0$ lorsque $g$ tend vers l'infini dans $G$. Michael Cowling \cite{Cowling} montre en fait que si $G$ a la propri\'et\'e $(T)$ il existe un entier $p=p(G) < +\infty$ tel que si $\pi$ est une repr\'esentation unitaire de $G$ sans vecteurs invariants alors $p(\pi) \leq p$, voir aussi \cite[\S 7]{Oh}. Si $G$ est localement isomorphe \`a $\mathrm{SO}(n,1)$ ou $\mathrm{SU} (n,1)$, ce n'est plus vrai mais la classification des repr\'esentations irr\'eductibles unitaires de ces groupes \cite{Hirai,Kraljevic} implique que si $\pi$ est une repr\'esentation unitaire de $G$
qui a un trou spectral, alors $p(\pi) < +\infty$.\footnote{Toute la puissance de la classification des repr\'esentations irr\'eductibles unitaires n'est pas utile ici. Il suffit de consid\'erer la partie de $\pi$ form\'ee des repr\'esentations sph\'eriques.} Le fait suivant est bien connu; on peut en trouver une d\'emonstration dans \cite[Lemma 3]{Bekka}.

\begin{prop} \label{P:sg}
Soit $\Gamma$ un r\'eseau dans $G$. La repr\'esentation r\'eguli\`ere droite $\rho_{\Gamma}^0$ de $G$ dans le sous-espace $L^2_0 (\Gamma \backslash G)$ de $L^2 (\Gamma \backslash G)$ orthogonal
aux fonctions constantes a un trou spectral. 
\end{prop}

 \'Etant donn\'e un \'el\'ement $g=nak \in G$, on pose $H(g) = a$. On appelle {\it fonction d'Harish-Chandra} la fonction $\Xi= \Xi_G : G \rightarrow \R$ d\'efinie par 
$$\Xi (g) = \int_K \rho (H(kg^{-1}))^{-1/2} dk \quad \mbox{o\`u } \rho (a) = \det \- {}_{\mathfrak{n}} (\mathrm{Ad} (a^{-1})).$$
La fonction $\Xi$ d\'ecro\^{\i}t exponentiellement vite le long de $A^+$~: modulo un facteur logarithmique, on a $\Xi (a) \asymp \rho (a)^{-1/2}$. Puisque la mesure de Haar $dg$ est \'egale \`a $\rho (a) dn da dk$, on en d\'eduit en particulier que $\Xi \in L^{2+\varepsilon} (G)$ pour tout r\'eel $\varepsilon >0$.

\begin{exem} Si $G= \PSL_2 (\C)$, on a, pour les choix usuels de groupes $K$, $A$, $N$,
\begin{equation*}
\begin{split}
\Xi (a_r) & = \frac{1}{\pi} \int_{0}^{\pi/2} (e^{-r} \cos^2 \theta + e^{r} \sin^2 \theta )^{-1} \sin (2 \theta ) d \theta \\ 
& = \frac{r}{\pi \sinh r}.
\end{split}
\end{equation*}
\end{exem}

Notons $d$ la dimension de $K$ et fixons une base $\mathcal{B}$ de l'alg\`ebre de Lie $\mathfrak{k}$ de $K$. \'Etant donn\'e une repr\'esentation unitaire
$\pi$ et un vecteur $v \in \mathcal{H}_{\pi}^{\infty}$, on pose 
$$S (v) = \sum_{\mathrm{ord} (D) \leq d+1} ||\pi (D) v ||,$$
o\`u $D$ parcourt l'ensemble des mon\^omes en les \'el\'ements de $\mathcal{B}$ de degr\'e $\leq d+1$ et, si $X_1, \ldots , X_r$ sont des \'el\'ements de $\mathcal{B}$, on a
$\pi (X_1 \cdots X_r)= \pi (X_1) \cdots \pi (X_r)$ et chaque $\pi(X_i)$ op\`ere par d\'erivation.

La proposition suivante -- qui est essentiellement due \`a Michael Cowling, Uffe Haagerup et Roger Howe \cite{CHH} -- appliqu\'ee \`a la repr\'esentation $\rho_{\Gamma}^0$ implique finalement le th\'eor\`eme \ref{ME}. 
\begin{prop}
Soit $\pi$ une repr\'esentation unitaire de $G$ telle que $p(\pi) \leq 2k$ avec $k \in \N^*$. Il existe une constante $C=C(G,k)$ telle que pour tous $v,w \in \mathcal{H}_{\pi}^{\infty}$ et
pour tout $g\in G$, on a~:
\begin{equation} \label{mc}
|\langle \pi (g) v , w \rangle | \leq C  S (v) S (w) \Xi^{1/k} (g).
\end{equation}
\end{prop}
\begin{proof} Quitte \`a remplacer $\pi$ par le produit tensoriel $\pi^{\otimes k}$ on peut supposer que $k=1$; voir \cite[p. 108]{CHH}. Il d\'ecoule alors de \cite[Thm. 1]{CHH} que $\pi$
est faiblement contenue dans dans la repr\'esentation r\'eguli\`ere (droite) $L^2 (G)$. On se ram\`ene alors facilement  \`a d\'emontrer la proposition dans le seul cas o\`u $\pi$
est la repr\'esentation r\'eguli\`ere de $G$ (et $k=1$); voir la d\'emonstration de \cite[Thm. 2]{CHH} pour plus de d\'etails sur cette derni\`ere r\'eduction.

Soient $v$ et $w$ dans  $L^2 (G) \cap C^{\infty} (G)$. Les fonctions $\varphi : x \mapsto \sup_{k \in K} |v (x k)|$ et 
$\psi : x \mapsto \sup_{k \in K} |w (xk)|$ sont positives, $K$-invariantes et $| \langle \pi (g) v , w \rangle | \leq  \int_G \varphi (xg) \psi (x) dx$. Mais il d\'ecoule du lemme de Sobolev qu'il existe une
constante $C$ telle que pour tout $x \in G$, 
$$\varphi (x)^2 \leq C \sum_{\mathrm{ord} (D) \leq d+1}  ||(\pi (D)v)(x \cdot ) ||^2_{L^2 (K)}.$$
En particulier $||\varphi || \leq \sqrt{C} S (v)$ et de m\^eme pour $\psi$. Il nous reste donc \`a v\'erifier que si $\varphi, \psi \in L^2 (G)$
sont des fonctions positives, $K$-invariantes et de norme $1$, alors $| \langle \pi (g) \varphi , \psi \rangle | \leq \Xi (g)$. C'est le calcul d\'etaill\'e p. 106--107 de \cite{CHH} que nous reprenons ici:
\begin{equation*}
\begin{split}
| \langle \pi (g) \varphi , \psi \rangle | & = \int_K \left( \int_{NA} \varphi (na) \psi (nakg^{-1}) \rho (a) dn da \right) dk \\
& \leq || \varphi || \int_K \left( \int_{NA} \psi (naH(kg^{-1}))^2 \rho (a) dn da  \right)^{1/2} dk \\
&  =  ||\varphi || \cdot || \psi || \int_K \rho (H(kg^{-1} ))^{-1/2} dk ,
\end{split}
\end{equation*}
o\`u l'on a utilis\'e l'in\'egalit\'e de Cauchy-Schwarz dans $L^2 (NA)$ et la $K$-invariance de $\varphi$ et $\psi$.
\end{proof}

\medskip

{\it Un grand merci \`a Laurent Clozel, Gilles Courtois, Ruben Dashyan, Bertrand Deroin, Olivier Guichard, Antonin Guilloux, Fr\'ed\'eric Haglund, H\'el\`ene Eynard-Bontemps, Elisha Falbel, Fran\c{c}ois Labourie, 
Fr\'ed\'eric Paulin et Maxime Wolff pour leur aide dans l'\'elaboration de ce texte.}

\bibliography{bibli}

\bibliographystyle{smfplain}

\end{document}